\UseRawInputEncoding
\documentclass[a4paper, 10pt,reqno]{amsart}
\usepackage{amsmath,amssymb,amsthm,enumerate}%,backref,cite
\allowdisplaybreaks
\theoremstyle{plain}
\newtheorem{theorem}{Theorem}
\newtheorem*{theorem*}{Theorem}
\newtheorem{lemma}{Lemma}
\newtheorem*{lemma*}{Lemma}
\theoremstyle{definition}
\newtheorem{definition}{Definition}
\newtheorem*{definition*}{Definition}
\theoremstyle{remark}
\newtheorem{remark}{Remark}
\newtheorem*{remark*}{Remark}
\newtheorem{statement}{Statement}
\newtheorem*{statement*}{Statement}

\begin{document}
\title[Singular distributions and fractals]{Certain singular distributions and fractals}
\author{Symon Serbenyuk}

\address{
  45~Shchukina St. \\
  Vinnytsia \\
  21012 \\
  Ukraine}
\email{simon6@ukr.net}

\subjclass[2010]{05D99, 11K55, 11J72, 28A80,  26A09}

% Key words
\keywords{
Fractal;  Moran structure; Hausdorff dimension; continuous function; range of values of a function.}

\begin{abstract}
In the present article, the main attention is given to fractal sets  whose elements have certain restrictions on
using digits or combinations of digits in own nega-P-representation.  Topological, metric, and fractal properties of images of certain self-similar fractals under the action of some singular distributions, are investigated. 
\end{abstract}
\maketitle

%%%%%%%%%%%%%%%%%%%%%%%%%%%%%%%%%%%%%%%%%%%%%%%%%%%%%%%%%%%%%%%%%%%%%%%%

\section{Introduction}

Nowadays, ``pathological" mathematical objects (the notion of "pathology" in mathematics is described in \cite{wiki-pathological}) such as fractals, functions with complicated (complex) local structure (i.e., singular, non-differentiable, or  nowhere  monotonic functions), and other mathematical objects,  have the applied importance and the interdisciplinary character. A number of researches are devoted to this topic (for example, see \cite{{Falconer1997}, {Falconer2004}, {Katsuura1991}}, \cite{wiki-fractal} - \cite{wiki-function}, etc.).

Fractals are widely applicated  in computer design,  quantum mechanics, solid-state physics, algorithms of the compression to information, analysis and categorizations of signals of  various forms appearing in different areas (e.g. the analysis of exchange rate fluctuations in economics),  etc. In addition, such sets are useful for checking  preserving the Hausdorff dimension by certain functions \cite{{S. Serbenyuk abstract1},{S.Serbenyuk 2017}}. 

In the present article, the main attention is given to fractals having the Moran structure. Fractal sets considering in this paper, are images of certain fractals under the map which is a some generalization of the Salem function. Thus, in the present research, metric, topological, and fractal properties of certain sets (images of certain fractals under the  some singular distribution) are investigated. Also,  some fractal properties of the considered  singular distribution are studied more detail. In other words, the differences  between fractal properties of the considered images and their corresponding preimages are described.

Let us describe the notion of the Moran structure. We will consider two definitions of the notion of Moran sets. The first definition was given by Moran in the paper~\cite{Moran1946}, and the second definition is of Hua et al.~(\cite{HRW2000}).

\begin{definition}{(Definition of  Moran).}
Let us consider space $\mathbb R^n$. In~\cite{Moran1946}, P.~A.~P.~Moran introduced the following construction of sets and calculated the Hausdorff dimension of the limit set 
\begin{equation}
\label{eq: Cantor-like set}
E=\bigcap^{\infty} _{n=1}{\bigcup_{i_1,\dots , i_n\in A_{0,p}}{\Delta_{i_1i_2\ldots i_n}}}. 
\end{equation}
 Here $p$ is a fixed positive integer, $A_{0,p}=\{1, 2, \dots , p\}$, and sets $\Delta_{i_1i_2\ldots i_n}$ are basic sets  having  the following properties:
\begin{itemize}
\item any set $\Delta_{i_1i_2\ldots i_n}$ is closed and disjoint;
\item for any $i\in A_{0,p}$, the condition $\Delta_{i_1i_2\ldots i_ni}\subset\Delta_{i_1i_2\ldots i_n}$ holds;
\item 
$$
\lim_{n\to\infty}{d\left(\Delta_{i_1i_2\ldots i_n}\right)}=0, \text{where $d(\cdot)$ is the diameter of a set};
$$
\item each basic set is the closure of its interior;
\item at each level the basic sets do not overlap (their interiors are disjoint);
\item any basic set $\Delta_{i_1i_2\ldots i_ni}$ is geometrically similar to $\Delta_{i_1i_2\ldots i_n}$;
\item
$$
\frac{d\left(\Delta_{i_1i_2\ldots i_ni}\right)}{d\left(\Delta_{i_1i_2\ldots i_n}\right)}=\sigma_i,
$$
where $\sigma_i\in (0,1)$ for $i=\overline{1,p}$.
\end{itemize}

The Hausdorff dimension $\alpha_0$ of the set $E$ is the unique root of the following equation
\begin{equation}
\label{eq: self-similar set}
\sum^{p} _{i=1}{\sigma^{\alpha_0} _i}=1.
\end{equation}

It is easy to see that set \eqref{eq: Cantor-like set} is a Cantor-like set and is a self-similar fractal. The set $E$  is called \emph{the Moran set}.
\end{definition}

\begin{definition}{(Definition of  Hua et al.).} 
Let $(n_k)$ be a sequence of positive integers, $J\in\mathbb R^n$ be a compact set with nonempty interior, and $(\Phi_k)$ be a sequence of positive real vectors with
$\Phi_k=(\sigma_{k,1},\sigma_{k,2},\dots , \sigma_{k,n_k}),$
where $k\in\mathbb N$ and
$$
\sum^{n_k} _{j=1}{\sigma_{k,j}}<1.
$$
A  set  of the form
$$
E=\bigcap^{\infty} _{k=1}{\bigcup_{i_1,\dots , i_k\in A_{0,n_k}}{\Delta_{i_1i_2\ldots i_k}}},
$$
where $A_{0,n_k}=\{1, 2, \dots , n_k\}$, is called \emph{the Moran
set associated with the collection $F$}. 
Here
$$
F=\bigcup^{\infty} _{k=0}{F_k}=\bigcup^{\infty} _{k=0}{\left\{J_{\sigma}:=\Delta_{i_1i_2...i_k}: k\in\mathbb N, i_k\in\{1, 2, \dots , n_k\}      \right\}}
$$
\emph{The collection $F$ fulfills the Moran structure} provided it satisfies the following Moran Structure Conditions (MSC):

\begin{enumerate}
\item 
\label{property-1}
$J_{\varnothing}=J$.
\item
\label{property-2}
An arbitrary $J_{\sigma}$ is geometrically similar to $J$.
\item
\label{property-3}
 For any $i,j\in \{1,2, \dots , n_{k+1}\}$ such  that $i\ne j$, the conditions 
$$
\Delta_{i_1i_2\ldots i_ki}\subset\Delta_{i_1i_2\ldots i_k}, ~~~\Delta_{i_1i_2\ldots i_ki}\cap\Delta_{i_1i_2\ldots i_kj}=\varnothing
$$
hold.
\item
\label{property-4} For any $j\in \{1,2, \dots , n_{k+1}\}$,
$$
\frac{d\left(\Delta_{i_1i_2\ldots i_kj}\right)}{d\left(\Delta_{i_1i_2\ldots i_k}\right)}=\sigma_{k+1, j}.
$$
\end{enumerate}
The elements of $F_k$ are called \emph{the basic elements
of order k of the Moran set $E$}, and the elements of $F$ are called \emph{ the basic elements of the
Moran set $E$}.
\end{definition}

\begin{remark}
Let us note that the main difference between definitions of Moran  and Hua is Property~\ref{property-4} in MSC.
\end{remark}

Let $M=M(J,(n_k), (\Phi_k))$ be a class of Moran sets satisfying MSC~\ref{property-1}--\ref{property-4}. It is known that one can define a sequence $(\alpha_k)$, where $\alpha_k$ satisfies the equation
$$
\prod^{k} _{i=1}{\sum^{n_i} _{j=1}{\sigma^{\alpha_k} _{i,j}}}=1.
$$
Also, suppose that 
$$
\alpha_{*}=\liminf_{k\to\infty}{\alpha_k}, ~~~ \alpha^{*}=\limsup_{k\to\infty}{\alpha_k};
$$
$$
c_{*}=\inf_{i,j}{\sigma_{i,j}}, ~~~~c^{*}=\sup_{i,j}{\sigma_{i,j}}.
$$

The following statements are useful for studying the Hausdorff dimension of sets investigated in the present research. Suppose that $\dim_HE$ is the Hausdorff dimension of the set $E$.

\begin{theorem}{(Theorem 1.1 in \cite{HRW2000}).}
\label{th: auxiliary-1}
Let $M=M(J, (n_k), (\Phi_k))$ be a Moran class satisfying $c_{*}=\inf{\sigma_{i,j}}>0$, then for any $E\in M$,
$$
\dim_HE=\alpha_{*}, ~\text{and}~ E~ \text{is a $s$-set if and only if}~ 0<\liminf_{k\to\infty}{\sum_{i_1,i_2,...,i_k}{d(\Delta_{i_1i_2...i_k})^{\alpha^{*}}}}<\infty.
$$
\end{theorem}

\begin{theorem}{(Theorem 1.2 in \cite{HRW2000}).}
\label{th: auxiliary-2}
Let $M=M(J, (n_k), (\Phi_k))$ be a Moran class. Suppose that the sequences $(n_k), (\Phi_k)$ satisfy the following conditions:
\begin{itemize}
\item $\sup_k{n_k}=\lambda<\infty$;
\item $0<\inf_i{\max_j{\{\sigma_{i,j}\}}}\le\sup_i{\max_j{\{\sigma_{i,j}\}}}<1$.
\end{itemize}
Then for all $E\in M$, $\dim_HE=\alpha_{*}$.
\end{theorem}

\begin{theorem}{(Theorem 1.3 in \cite{HRW2000}).}
\label{th: auxiliary-3}
Let $M=M(J, (n_k), (\Phi_k))$ be a Moran class. Suppose
$$
\lim_{k\to\infty}{\frac{\log{d_k}}{\log{M_k}}}=0,
$$
where $d_k:=\min_{1\le j\le n_k}{\sigma_{k,j}}$, $M_k:=\max_{i_1,i_2,...,i_k}{d(\Delta_{i_1i_2...i_k})}$. Then for all $E\in M$, $\dim_HE=\alpha_{*}$.
\end{theorem}

\begin{theorem}{(Theorem 3.1 in \cite{HRW2000}).}
\label{th: auxiliary-4}
Let $c_{*}=\inf{\sigma_{i,j}}>0$. Then for any $E\in M(J, (n_k), (\Phi_k))$, we have  
$$
\dim_HE=\alpha_{*}.
$$
\end{theorem}

Let us return to the description of  investigations of the present paper. 

Let $s>1$ be a fixed positive integer. Then \emph{the s-adic representation of numbers from~$[0,1]$} is a representation of the following form:
$$
x=\Delta^s _{\alpha_1\alpha_2...\alpha_n...}=\sum^{\infty} _{n=1}{\frac{\alpha_n}{s^n}},
$$
where $\alpha_n\in A=\{0,1,\dots, s-1\}$.

In addition,  the following representation
$$
x=\Delta^{-s }_{\alpha_1\alpha_2...\alpha_n...}=\sum^{\infty} _{n=1}{\frac{\alpha_n}{(-s)^n}},
$$
is \emph{the nega-s-adic representation of numbers from $\left[-\frac{s}{s+1}, \frac{1}{s+1}\right]$}. Here $\alpha_n\in A$ as well.

It is easy to see that 
\begin{equation*}
%\label{eq: relationship1}
x=\Delta^{-s }_{\alpha_1\alpha_2...\alpha_n...}\equiv \frac{1}{s+1}-\Delta^{s }_{\alpha_1[s-1-\alpha_2]...\alpha_{2k-1}[s-1-\alpha_{2k}]...}
\end{equation*}
or
\begin{equation*}
x=\Delta^{-s }_{\alpha_1\alpha_2...\alpha_n...}\equiv \Delta^{s }_{\alpha_1[s-1-\alpha_2]...\alpha_{2k-1}[s-1-\alpha_{2k}]...}-\frac{s}{s+1}.
\end{equation*}

Let us consider  the sets 
\begin{equation*}
\mathbb S_{(s,u)}=\left\{x: x= \Delta^{s}_{{\underbrace{u\ldots u}_{\alpha_1-1}} \alpha_1{\underbrace{u\ldots u}_{\alpha_2 -1}}\alpha_2 ...{\underbrace{u\ldots u}_{ \alpha_n -1}}\alpha_n...},   \alpha_n \in A_0=\{1,2,\dots , s-1\}\setminus\{u\} \right\}
\end{equation*}
and
\begin{equation*}
\mathbb S_{(-s,u)}=\left\{x: x= \Delta^{-s}_{{\underbrace{u\ldots u}_{\alpha_1-1}} \alpha_1{\underbrace{u\ldots u}_{\alpha_2 -1}}\alpha_2 ...{\underbrace{u\ldots u}_{ \alpha_n -1}}\alpha_n...},  \alpha_n \in A_0=\{1,2,\dots , s-1\}\setminus\{u\}  \right\}
\end{equation*}
where $u=\overline{0,s-1}$, the parameters $u$ and $2<s\in\mathbb N$ are fixed for the fixed sets $\mathbb  S_{(s,u)}$, $\mathbb  S_{(-s,u)}$.

 Elements of these sets have certain  restrictions on using combinations of digits in own  representations. For example, let $s>2$ and $u\in \{0,1, \dots , s-1\}$ be fixed positive integers. Then the set $\mathbb S_{(s,u)}$ is the set whose elements represented in terms of the s-adic representation and contain combinations of s-adic digits only from the set
$$
\left\{1, u2, uu3, \dots , \underbrace{u\ldots u}_{u -2}[u-1], \underbrace{u\ldots u}_{u}[u+1], \dots, \underbrace{u\ldots u}_{s -3}[s-2]\}, \underbrace{u\ldots u}_{s -2}[s-1]\}\right\}. 
$$
In the general case, we have a class  $\Upsilon_s$ of the sets  $\mathbb  S_{(s,u)}$ represented  in the form 
\begin{equation*}
\mathbb S_{(s,u)}= \left\{x: x=\frac{u}{s-1} +\sum^{\infty} _{n=1} {\frac{\alpha_n - u}{s^{\alpha_1+\dots+\alpha_n}}},  \alpha_n \ne u, \alpha_n \ne 0  \right\},
\end{equation*}
where  $u=\overline{0,s-1}$ and   parameters $u$,  $s$ are fixed for the set $\mathbb  S_{(s,u)}$. That is, $\Upsilon_s$ contains the sets  $\mathbb  S_{(s,0)}, \mathbb  S_{(s,1)},\dots, \mathbb  S_{(s,u-1)}, \mathbb  S_{(s,u+1)}, \dots, \mathbb  S_{(s,s-1)}$. We say that  a class  $\Upsilon$ of sets  contains  $\Upsilon_3, \Upsilon_4,\dots ,\Upsilon_n,\dots$.

Some articles (for example, see \cite{ DU2014, DU2014(2), {S. Serbenyuk fractals},{S. Serbenyuk abstract 2}, {S. Serbenyuk abstract 3},{S. Serbenyuk abstract 5}, {Symon1}, {Symon2}, {S. Serbenyuk 2013}, {S. Serbenyuk 2017  fractals}} ) are devoted to   sets whose elements have certain  restrictions on using combinations of digits in own s-adic  representation. 

Let us discuss properties of the last-mentioned sets. 

\begin{theorem}[\cite{{S. Serbenyuk 2017  fractals}}\footnote{See Lemmas~1 and~2, Theorems 1--3 which were published   with proofs in English in the preprint \cite{{S. Serbenyuk 2017  fractals}}. These results were  published in the papers~\cite{Symon1, Symon2} in Ukrainian.}; \cite{S. Serbenyuk   fractals}\footnote{This preprint contains results translated into English without proofs from \cite{Symon1, Symon2, {S. Serbenyuk 2013}}. See Theorems 4, 6, and 8 in \cite{S. Serbenyuk   fractals}.}; \cite{Symon1, {Symon2}, {S. Serbenyuk 2013}}.]
\label{th: theorem1}
For  an arbitrary $u \in A$, the sets $\mathbb S_{(s,u)}$ and $\mathbb S_{(-s,u)}$ are  uncountable,   perfect,   nowhere dense sets of zero Lebesgue measure, and  self-similar fractals whose Hausdorff dimension $\alpha_0$ satisfies the following equation 
$$
\sum _{p \ne u, p \in A_0} {\left(\frac{1}{s}\right)^{p \alpha_0}}=1.
$$
\end{theorem}

\begin{remark}
We note that the statement of the last-mentioned theorem is true for all sets $\mathbb  S_{(\pm s,0)}, \mathbb  S_{(\pm s,1)},\dots,\mathbb  S_{(\pm s,s-1)}$ (for  fixed parameters  $u=\overline{0,s-1}$ and any fixed $2<s\in\mathbb N$ )  without the sets $ S_{(\pm 3,1)}$ and $ S_{(\pm 3,2)}$. 
\end{remark}

\begin{remark}
Some properties of $\mathbb S_{(s,u)}$ and $\mathbb S_{(-s,u)}$ are identical (it follows from the last theorem). However (\cite{{S. Serbenyuk 2013}}), certain local properties, e.g., a disposition of cylinders, etc., are different. Also, for the case of $\mathbb S_{(-s,u)}$, proofs are more difficult. 
\end{remark}

\begin{remark}
The sets $\mathbb  S_{(3,0)}$, $\mathbb  S_{(-3,0)}$, and also sets $\mathbb S_{(s,u)}$ and $\mathbb S_{(-s,u)}$ are Cantor-like sets, Moran sets,  and self-similar fractals for any  positive integer $s>3$ and an arbitrary integer $0\le u<s$.  Really, these sets  have structure \eqref{eq: Cantor-like set}, their Hausdorff dimensions are   solutions  of corresponding  equations \eqref{eq: self-similar set}, and their cylinder sets are sets having  properties of basic sets $\Delta_{i_1i_2...i_n}$.
\end{remark}

 Cantor-like sets are important and appear in a number of researches in various areas of mathematics. We present some of them. For example, such sets are important in the study of Diophantine approximation. In \cite{Dani2012},  it is proven that  a large class of Cantor-like sets of $\mathbb R^d, d \ge 1$, contains uncountably many badly approximable numbers, respectively badly approximable
vectors, when $d \ge 2$. In 1984, Kurt Mahler posed the following fundamental question: How
well can irrationals in the Cantor set be approximated by rationals in the Cantor set? In~\cite{Broderick2011}, towards such a theory,  a Dirichlet-type theorem for this intrinsic diophantine approximation on Cantor-like sets was proven. ``The resulting approximation function is analogous to that for $\mathbb R^d$, but with $d$ being the Hausdorff dimension of the set, and logarithmic dependence on the denominator instead". Let us note that the Cantor set is a set whose elements have a restriction on using ternary digits. This set is the best known example of a fractal in the real line. 

``Nowadays, the ternary Cantor set is the paradigmatic model of the fractal geometry \cite{Mandelbrot1999, Taylor2012}  and in many branches of physics  (see \cite{Bunde1994}). A large class of Cantor-type sets frequently appear as invariant sets
and attractors of many dynamical systems of the real world problems, see \cite{Kennedy1995}."~(\cite{TSBR2017})

The papers \cite{BBLS2016, TSBR2017}, etc., are devoted to the study
of Cantor-type sets in hyperbolic numbers.  The study of arithmetical sum of Cantor-type sets
plays a special role in dynamical systems (explanations are given in \cite{Palis1993}). 
 
To indicate preserving the Hausdorff dimension by singular distribution, Cantor-like sets are useful. For example, in \cite{S.Serbenyuk 2017}, certain self-similar fractals defined in terms of the nega-binary representation and whose elements have some restriction on using binary digits were used for proving a fact that the Minkowski function does not preserve the Hausdorff dimension.

 Finally, let us remark that restrictions on using elements of sets $\mathbb S_{(\pm s,u)}$ are new (they occur for the first time). 

Let $s>1$ be a fixed positive integer and $\alpha_n\in A=\{0,1,\dots, s-1\}$. Let  $P=\{p_0,p_1,\dots , p_{s-1}\}$ be a fixed set whose elements satisfy the following properties: $p_0+p_1+\dots+p_{s-1}=1$ and $p_i>0$ for all $i=\overline{0,s-1}$. Then let us consider the following distribution functions. 

Let $\zeta$ be a random variable defined by the s-adic representation
$$
\zeta= \frac{\iota_1}{s}+\frac{\iota_2}{s^2}+\frac{\iota_3}{s^3}+\dots +\frac{\iota_{k}}{s^{k}}+\dots   = \Delta^{s} _{\iota_1\iota_2...\iota_{k}...},
$$
where  digits $\iota_k$ $(k=1,2,3, \dots )$ are random and taking the values $0,1,\dots ,s-1$ with positive probabilities ${p}_{0}, {p}_{1}, \dots , {p}_{s-1}$. That is,  $\iota_k$ are independent and  $P\{\iota_k=\alpha_k\}={p}_{\alpha_k}$, $\alpha_k \in A$. 

Let $\varsigma$ be a random variable defined by the s-adic representation
$$
\varsigma= \Delta^{s} _{\pi_1\pi_2...\pi_{k}...}=\sum^{\infty} _{k=1}{\frac{\pi_k}{s^k}},
$$
where 
$$
\pi_k=\begin{cases}
\alpha_k&\text{if  $k$ is odd}\\
s-1-\alpha_k&\text{if $k$ is even}
\end{cases}
$$
and digits $\pi_k$ $(k=1,2,3, \dots )$ are random and taking the values $0,1,\dots ,s-1$ with positive probabilities ${p}_{0}, {p}_{1}, \dots , {p}_{s-1}$. That is, $\pi_k$ are independent and  $P\{\pi_k=\alpha_k\}={p}_{\alpha_k}$, $P\{\pi_k=s-1-\alpha_k\}={p}_{s-1-\alpha_k}$, where $\alpha_k \in A$.

Let us consider the distribution function  ${f}_{\zeta}$ of the random variable $\zeta$ and the distribution function  ${\tilde F}_{\varsigma}$ of the random variable $\varsigma$ (their particular form will be noted in this paper  later).

Let $x\in\mathbb S_{(s,u)}$. Let us consider the image 
$$
\tilde y=\tilde F \circ f_l\circ f_+(x),
$$
where 
$$
f_+: x=\Delta^s _{\alpha_1\alpha_2...\alpha_n...}\to \Delta^{-s} _{\alpha_1\alpha_2...\alpha_n...}=y
$$
 is not monotonic on the domain and is a nowhere differentiable function~(\cite{S. Serbenyuk functions with complicated local structure 2013}), $f_l(y)=\frac{1}{s+1}-y$, and $\tilde F$ is the last-mentioned distribution function.  That is, in this paper, the main attention is given to properties and a  local structure of a set of the form:
$$
\mathbb S_{(-P,u)}=\left\{\tilde y: \tilde y=\tilde F \circ f_l\circ f_+(x), x\in\mathbb S_{(s,u)} \right\}
$$ 
It is easy to see that 
$$
\{z: z=\tilde F\circ f_l (x),~~~ x\in \mathbb S_{(-s,u)}\}\equiv \mathbb S_{(-P,u)}.
$$

Finally, one can note that 
$$
\mathbb S_{(P,u)}=\{y: y=f_{\xi}(x), x\in \mathbb S_{(s,u)}\}.
$$
More considerations about the last images of $\mathbb S_{(s,u)}$ and $\mathbb S_{(-s,u)}$ are given in one of the next sections of this paper. 

There are much research devoted to Moran-like constructions and Cantor-like sets (for example, see \cite{{Falconer1997},{Falconer2004}, {Mandelbrot1977},  {PS1995}, DU2014, DU2014(2),  HRW2000, PS1995,  {S.Serbenyuk 2017}, {S. Serbenyuk fractals}} and references therein).

Moran sets and homogeneous Moran sets are very important fractal sets. There exist  many important applications of such sets  in multifractal analysis and  studying  the structure of the spectrum of quasicrystals  (see \cite{KLS2016, LW2011, W2005} and references therein), in the power systems (\cite{Feng2005} and its reference)  and measurement of number theory (see \cite{WW2008, Wu2005}), etc. This facts induce the great research interest in modeling such sets (in terms of various numeral systems) and  studying properties of  their images under certain maps.

Let us consider the main properties of the set $\mathbb S_{(P,u)}$.
\begin{theorem}{\cite{Symon}.}
The set $\mathbb S_{(P,u)}$ is an uncountable, perfect, and nowhere dense set of zero Lebesgue measure.
\end{theorem}
\begin{theorem}{\cite{Symon}.}
The set  $\mathbb S_{(P,u)} $ is a self-similar fractal whose Hausdorff dimension $\alpha_0 (\mathbb S_{(P,u)})$  satisfies the following equation 
$$
\sum _{i\in A_0\setminus\{u\}} {\left(p_ip^{i-1} _u\right)^{\alpha_0}}=1.
$$
\end{theorem}

 In the present article, the main attention is given to properties of sets $\mathbb S_{(-P,u)}$. 
In this paper, topological and metric, local and  fractal properties of the set $\mathbb S_{(-P,u)}$ are investigated. 

\begin{remark}
One can note that  the present investigations are similar with investigations (\cite{Symon}) for the set $\mathbb S_{(P,u)}$,  but are more complicated and some techniques for proving the main statements are different. In addition, one can note that  if for $\mathbb S_{(-s,u)}$ and $\mathbb S_{(s,u)}$,  topological,  metric, and  fractal properties (without some properties of cylinders) are similar, then   fractal  and some local  properties of $\mathbb S_{(-P,u)}$ and $\mathbb S_{(P,u)}$ are different. For example, $\mathbb S_{(P,u)}$ is a self-similar fractal (i.e., this is a Moran set by Moran's definition, \cite{Moran1946}) but $\mathbb S_{(-P,u)}$ is  a non-self-similar set having the Moran structure (i.e., this is a  Moran set by the definition of Hua et al. (see the  definition in \cite{HRW2000})). 
\end{remark}

Let us formulate the main results of the present paper.

Suppose $d(\cdot)$ is the diameter of a set and a cylinder $\Delta^{(-P,u)} _{c_1c_2...c_n}$ is a set whose elements are elements of $\mathbb S_{(-P,u)}$ and for these elements the condition $\alpha_i=c_i$ holds for all $i=\overline{1,n}$ (here $c_1, c_2,\cdots, c_n$ is a fixed tuple). The main resuts of the present paper are formulated in the following theorems. 

\begin{theorem}
\label{th: the first main theorem}
An arbitrary set $\mathbb S_{(-P,u)}$ is an uncountable, perfect, and nowhere dense set of zero Lebesgue measure.
\end{theorem}

One can note that for proving topological and metric properties of $\mathbb S_{(-P,u)}$,  cylinders of of this set and their properties are used. Such investigations includes local properties of the set $\mathbb S_{(-P,u)}$.

\begin{theorem}
\label{th: the second main theorem}
In the general case,  the set  $\mathbb S_{(-P,u)} $ is not a self-similar fractal,  the Hausdorff dimension $\alpha_0 (\mathbb S_{(-P,u)})$ of which can be calculated by the formula: 
$$
\alpha_0=\liminf_{k\to\infty}{\alpha_k},
$$
where $(\alpha_k)$ is a sequence of numbers satisfying the equation
$$
\left(\sum_{\substack{c_1 \text{is odd}\\ c_1\in \overline{A}}}{\left(\omega_{2,c_1}\right)^{\alpha_1} }+\sum_{\substack{c_1 \text{is even}\\ c_1\in \overline{A}}}{\left(\omega_{4,c_1}\right)^{\alpha_1} }\right)\times
$$
$$
\times\prod^k _{i=2}{\left(\sum_{\substack{c_i \text{is odd}\\ c_i\in \overline{A}}}{N_{1,c_i}\left(\omega_{1,c_i}\right)^{\alpha_i} }+\sum_{\substack{c_i \text{is odd}\\ c_i\in \overline{A}}}{N_{2,c_i}\left(\omega_{2,c_i}\right)^{\alpha_i} }+\sum_{\substack{c_i \text{is even}\\ c_i\in \overline{A}}}{N_{3,c_i}\left(\omega_{3,c_i}\right)^{\alpha_i} }+\sum_{\substack{c_i \text{is even}\\ c_i\in \overline{A}}}{N_{4,c_i}\left(\omega_{4,c_i}\right)^{\alpha_i} }\right)} =1.   
$$
Here $N_{j,c_i}$ ($j=\overline{1,4}, 1<i\in\mathbb N$) is the number of cylinders $\Delta^{(-P,u)} _{c_1c_2...c_i}$ for which
$$
\frac{d\left(\Delta^{(-P,u)} _{c_1c_2...c_{i-1}c_i}\right)}{d\left(\Delta^{(-P,u)} _{c_1c_2...c_{i-1}}\right)}=\omega_{j,c_i}.
$$
Also, 
$$
\omega_{1,c_i}=\underbrace{p_{s-1-u}p_u\ldots p_{s-1-u}p_u}_{c_i-1}p_{s-1-c_i}\frac{d\left(\overline{\mathbb S_{(P,u)}}\right)}{d\left(\underline{\mathbb S_{(P,u)}}\right)}~~~\text{for an  odd number $c_i$},
$$
$$
\omega_{2,c_i}=\underbrace{p_up_{s-1-u}\ldots p_up_{s-1-u}}_{c_i-1}p_{c_i}\frac{d\left(\underline{\mathbb S_{(P,u)}}\right)}{d\left(\overline{\mathbb S_{(P,u)}}\right)}~~~\text{for an odd number $c_i$},
$$
$$
\omega_{3,c_i}=\underbrace{p_{s-1-u}p_u\ldots p_{s-1-u}p_up_{s-1-u}}_{c_i-1}p_{c_i}~~~\text{for an even number $c_i$},
$$
$$
\omega_{4,c_i}=\underbrace{p_up_{s-1-u}\ldots p_up_{s-1-u}p_{u}}_{c_i-1}p_{s-1-c_i}~~~\text{for an even number $c_i$}.
$$
In addition, $N_{1,c_i}+N_{2,c_i}=l(m+l)^{i-1}$ and $N_{3,c_i}+N_{4,c_i}=m(m+l)^{i-1}$, where $l$ is the number of odd numbers in the set  $\overline{A}=A\setminus\{0,u\}$ and $m$ is the number of even numbers in $\overline{A}$.
\end{theorem}

%%%%%%%%%%%%%%%%%%%%%
\section{The notion of the Salem type functions}
%%%%%%%%%%%%%%%%%%%%%

A function of the form 
$$
f(x)=f\left(\Delta^2 _{\alpha_1\alpha_2...\alpha_n...}\right)=\beta_{\alpha_1}+ \sum^{\infty} _{n=2} {\left(\beta_{\alpha_n}\prod^{n-1} _{i=1}{p_i}\right)}=y=\Delta^{P_2} _{\alpha_1\alpha_2...\alpha_n...},
$$
where $p_0>0$, $p_1>0$, and $p_0+p_1=1$, is called \emph{the Salem function}. Here $\beta_0=0, \beta_1=p_0$. This is one of the simplest examples of singular functions but its generalizations can be non-differentiable functions or do not have the derivative at points  from  a certain set.  In \cite{Salem1943}, Salem modeled this function for the case when the argument  represented in terms of the $s$-adic representation. 
Let us note that many researches are devoted to the Salem function and its generalizations (for example, see \cite{ACFS2017, Kawamura2010, Symon2015, Symon2017, Symon2019} and references in these papers). 

Let us consider a technique for modeling the Salem type function. Such functions are  the main functions  of the present  investigation. 

Let $\zeta$ be a random variable defined by the s-adic representation
$$
\zeta= \frac{\iota_1}{s}+\frac{\iota_2}{s^2}+\frac{\iota_3}{s^3}+\dots +\frac{\iota_{k}}{s^{k}}+\dots   = \Delta^{s} _{\iota_1\iota_2...\iota_{k}...},
$$
where  digits $\iota_k$ $(k=1,2,3, \dots )$ are random and taking the values $0,1,\dots ,s-1$ with positive probabilities ${p}_{0}, {p}_{1}, \dots , {p}_{s-1}$. That is,  $\iota_k$ are independent and  $P\{\iota_k=\alpha_k\}={p}_{\alpha_k}$, $\alpha_k \in A$.

From the definition of a distribution function and the following expressions 
$$
\{\zeta<x\}=\{\xi_1<\alpha_1(x)\}\cup\{\iota_1=\alpha_1(x),\iota_2<\alpha_2(x)\}\cup \ldots 
$$
$$
\cup\{\iota_1=\alpha_1(x),\iota_2=\alpha_2(x),\dots ,\iota_{k-1}=\alpha_{k-1}(x), \iota_{k}<\alpha_{k}(x)\}\cup \dots,
$$
$$
P\{\iota_1=\alpha_1(x),\iota_2=\alpha_2(x),\dots ,\iota_{k-1}=\alpha_{k-1}(x), \iota_{k}<\alpha_{k}(x)\}
=\beta_{\alpha_{k}(x)}\prod^{k-1} _{j=1} {{p}_{\alpha_{j}(x)}},
$$
where
$$
\beta_{\alpha_k}=\begin{cases}
\sum^{\alpha_k(x)-1} _{i=0} {p_{i}(x)}&\text{whenever $\alpha_k(x)>0$}\\
0&\text{whenever  $\alpha_k(x)=0$,}
\end{cases}
$$
it is easy to see that the following statement is true.

\begin{statement}
The distribution function  ${f}_{\zeta}$ of the random variable $\zeta$ can be represented in the following form
$$
{f}_{\zeta}(x)=\begin{cases}
0&\text{whenever $x< 0$}\\
\beta_{\alpha_1(x)}+\sum^{\infty} _{k=2} {\left({\beta}_{\alpha_k(x)} \prod^{k-1} _{j=1} {{p}_{\alpha_j(x)}}\right)}&\text{whenever $0 \le x<1$}\\
1&\text{whenever $x\ge 1$,}
\end{cases}
$$
where ${p}_{\alpha_{j(x)}}>0$.
\end{statement}

It is easy to see that
$$
x=\Delta^s _{\tilde\alpha_1\tilde\alpha_2...\tilde\alpha_k...}=\frac{1}{s+1}-\Delta^{-s} _{\alpha_1\alpha_2...\alpha_k...}\equiv\frac{1}{s+1}-\sum^{\infty} _{k=1}{\frac{(-1)^k\alpha_{k}}{s^{k}}}=\sum^{\infty} _{k=1}{\frac{\alpha_{2k-1}}{s^{2k-1}}}+\sum^{\infty} _{k=1}{\frac{s-1-\alpha_{2k}}{s^{2k}}}.
$$

Let us consider the following distribution function. 
By analogy, let $\varsigma$ be a random variable defined by the s-adic representation
$$
\varsigma= \Delta^{s} _{\pi_1\pi_2...\pi_{k}...}=\sum^{\infty} _{k=1}{\frac{\pi_k}{s^k}},
$$
where 
$$
\pi_k=\begin{cases}
\alpha_k&\text{if  $k$ is odd}\\
s-1-\alpha_k&\text{if $k$ is even}
\end{cases}
$$
and digits $\pi_k$ $(k=1,2,3, \dots )$ are random and taking the values $0,1,\dots ,s-1$ with positive probabilities ${p}_{0}, {p}_{1}, \dots , {p}_{s-1}$. That is, $\pi_k$ are independent and  $P\{\pi_k=\alpha_k\}={p}_{\alpha_k}$, $P\{\pi_k=s-1-\alpha_k\}={p}_{s-1-\alpha_k}$, where $\alpha_k \in A$. From the definition of a distribution function and the following expressions 
$$
\{\varsigma<x\}=\{\pi_1<\alpha_1(x)\}\cup\{\pi_1=\alpha_1(x),\pi_2<s-1-\alpha_2(x)\}\cup \ldots
$$
$$
\ldots \cup\{\pi_1=\alpha_1(x),\pi_2=s-1-\alpha_2(x),\ldots,\pi_{2k-1}<\alpha_{2k-1}(x)\}\cup
$$
$$
\cup\{\pi_1=\alpha_1(x),\pi_2=s-1-\alpha_2(x),\ldots,\pi_{2k-1}=\alpha_{2k-1}(x),\pi_{2k}<s-1-\alpha_{2k}(x)\}\cup\ldots,
$$
$$
P\{\pi_1=\alpha_1(x),\pi_2=s-1-\alpha_2(x),\ldots,\pi_{2k-1}<\alpha_{2k-1}(x)\}=\beta_{\alpha_{2k-1}(x)}\prod^{2k-2} _{j=1} {\tilde{p}_{\alpha_{j}(x)}}
$$
and
$$
P\{\pi_1=\alpha_1(x),\pi_2=s-1-\alpha_2(x),\dots,\pi_{2k}<s-1-\alpha_{2k}(x)\}=\beta_{s-1-\alpha_{2k}(x)}\prod^{2k-1} _{j=1} {\tilde{p}_{\alpha_{j}(x)}},
$$
we have the following statement. 
\begin{statement}
The distribution function  ${\tilde F}_{\varsigma}$ of the random variable $\varsigma$ can be represented in the following form
$$
{\tilde F}_{\varsigma}(x)=\begin{cases}
0&\text{whenever $x< 0$}\\
\tilde\beta_{\alpha_1(x)}+\sum^{\infty} _{k=2} {\left({\tilde\beta}_{\alpha_k(x)} \prod^{k-1} _{j=1} {{\tilde p}_{\alpha_j(x)}}\right)}&\text{whenever $0 \le x<1$}\\
1&\text{whenever $x\ge 1$,}
\end{cases}
$$
where ${p}_{\alpha_{j(x)}}>0$,
\end{statement}
$$
\tilde  p_{\alpha_k}=\begin{cases}
p_{\alpha_k}&\text{if  $k$ is odd}\\
p_{s-1-\alpha_k}&\text{if $k$ is even},
\end{cases}
$$
and
$$
\tilde\beta_{\alpha_k}=\begin{cases}
\beta_{\alpha_k}&\text{if  $k$ is odd}\\
\beta_{s-1-\alpha_k}&\text{if $k$ is even}.
\end{cases}
$$
The function  
$$ 
{\tilde F}(x)=\beta_{\alpha_1(x)}+\sum^{\infty} _{n=2} {\left({\tilde\beta}_{\alpha_n(x)}\prod^{n-1} _{j=1} {{\tilde p}_{\alpha_j(x)}}\right)},
$$
is a partial case of the function investigated in \cite{Symon2019}.

\begin{definition} 
Any function is said to be \emph{a function of the Salem type} if it is modeled by the mentioned technique. That is, the Salem type function is a distribution function of the  random variable $\eta=\Delta_{\xi_1\xi_2...\xi_k...}$ defined by a certain  representation $\Delta_{i_1i_2...i_k...}$ of real numbers, where digits $\xi_k$ $(k=1,2,3, \dots )$ are random and taking the values $0,1,\dots ,m_k$ with positive probabilities ${p}_{0}, {p}_{1}, \dots , {p}_{m_k}$ ($\xi_k$ are independent and  $P\{\xi_k=i_k\}={p}_{i_k}$, $i_k \in A_k=\{0,1,\dots, m_k\}$, $m_k\in \mathbb N\cup\{\infty\}$). In addition, some generalizations of the Salem type functions (for example, the last functions for which $p_i\in (-1,1)$) also are called \emph{functions of  the Salem type}. 
\end{definition}

By analogy, let $\eta$ be a random variable defined by the s-adic representation
$$
\eta= \Delta^{s} _{\xi_1\xi_2...\xi_{k}...}=\sum^{\infty} _{k=1}{\frac{\xi_k}{s^k}},
$$
where 
$$
\xi_k=\begin{cases}
\alpha_k&\text{if  $k$ is even}\\
s-1-\alpha_k&\text{if $k$ is odd}
\end{cases}
$$
and digits $\xi_k$ $(k=1,2,3, \dots )$ are random and taking the values $0,1,\dots ,s-1$ with positive probabilities ${p}_{0}, {p}_{1}, \dots , {p}_{s-1}$. Then the distribution function  ${\ddot F}_{\eta}$ of the random variable $\eta$ can be represented in the following form
$$
{\ddot F}_{\eta}(x)=\begin{cases}
0&\text{whenever $x< 0$}\\
\beta_{s-1-\alpha_1(x)}+\sum^{\infty} _{k=2} {\left({\ddot\beta}_{\alpha_k(x)} \prod^{k-1} _{j=1} {{\ddot p}_{\alpha_j(x)}}\right)}&\text{whenever $0 \le x<1$}\\
1&\text{whenever $x\ge 1$,}
\end{cases}
$$
where ${p}_{\alpha_{j(x)}}>0$,
$$
\ddot  p_{\alpha_k}=\begin{cases}
p_{\alpha_k}&\text{if  $k$ is even}\\
p_{s-1-\alpha_k}&\text{if $k$ is odd},
\end{cases}
$$
and
$$
\ddot\beta_{\alpha_k}=\begin{cases}
\beta_{\alpha_k}&\text{if  $k$ is  even}\\
\beta_{s-1-\alpha_k}&\text{if $k$ is odd}.
\end{cases}
$$

%%%%%%%%%%%%%%%%%%%%%%%%%%%%%%%%%%%%%%%%%%%%%%%
\section{$P$- and nega-$P$-representations}
%%%%%%%%%%%%%%%%%%%%%%%%%%%%%%%%%%%%%%%%%%%%%%%

Let us note that in  the last decades, there exists   the tendency to modeling and studying   numeral  systems defined in terms of alternating expansions of real numbers under the condition that numeral systems defined in terms of the corresponding positive expansions of real numbers are  known.  For example, a number of researchers investigate positive and alternating $\beta$-expansions, L\"uroth series, Engel series, etc. Let us remark that positive and alternating $\beta$-expansions are the  s-adic and nega-s-adic representations with $s=\beta>1$, $\beta\in\mathbb R$. The notion of $\beta$-expansions  was introduced by A.~R\'enyi in 1957 in the paper~\cite{Renyi1957}, but $(-\beta)$-expansions  were introduced in the  paper~\cite{Ito2009} in 2009. In addition, in 1883, the German mathematician J. L\"uroth introduced an expansion of a real number in the form of  a special series (now these series are called  L\"uroth series). However, in 1990, S. Kalpazidou, A. Knopfmacher, and J. Knopfmacher introduced
alternating L\"uroth series  in the paper~\cite{KKK1990}. 

Peculiarities of such numeral systems are the following: the similarity of some properties for mathematical objects defined in terms of corresponding positive and alternating expansions;   the complexity of proofs  of corresponding statements for the case of such alternating representations. Let us note that discovering differences in properties of fractals which are defined in terms of a certain type of postive and alternating expansions of real numbers ($P$- and nega-$P$-representations), is one of the aims of this paper.

The function
$$ 
{f}(x)=\beta_{\alpha_1(x)}+\sum^{\infty} _{n=2} {\left({\beta}_{\alpha_n(x)}\prod^{n-1} _{j=1} {{p}_{\alpha_j(x)}}\right)},
$$
can be used as a representation (\emph{the P-representation}) of numbers from $[0,1]$. That is,
$$
y=\Delta^{P} _{\alpha_1(x)\alpha_2(x)...\alpha_n(x)...}=f(x)=\beta_{\alpha_1(x)}+\sum^{\infty} _{n=2} {\left({\beta}_{\alpha_n(x)}\prod^{n-1} _{j=1} {{p}_{\alpha_j(x)}}\right)},
$$
where $P=\{p_0,p_1,\dots , p_{s-1}\}$, $p_0+p_1+\dots+p_{s-1}=1$, and $p_i>0$ for all $i=\overline{0,s-1}$.  In other words, this function ``preserves" digits of representations of numbers:
$$
f: x=\Delta^{s} _{\alpha_1(x)\alpha_2(x)...\alpha_n(x)...}\to \Delta^{P} _{\alpha_1(x)\alpha_2(x)...\alpha_n(x)...}=\Delta^{P} _{\alpha_1(y)\alpha_2(y)...\alpha_n(y)...}=y.
$$

We begin with  the nega-P- and  P-representations  of numbers from $[0,1]$:

\begin{equation}
\label{nega}
x=\sum^{\alpha_1-1} _{i=0}{p_{i}}+\sum^{\infty} _{n=2}{\left((-1)^{n-1}\tilde \delta_{\alpha_n}\prod^{n-1} _{j=1}{\tilde p_{\alpha_j}}\right)}+\sum^{\infty} _{n=1}{\left(\prod^{2n-1} _{j=1}{\tilde p_{\alpha_j}}\right)}\equiv\Delta^{-P} _{\alpha_1(x)\alpha_2(x)...\alpha_n(x)...}
\end{equation}
and 
\begin{equation}
\label{P}
x=\Delta^{P} _{\alpha_1(x)\alpha_2(x)...\alpha_n(x)...}=\beta_{\alpha_1(x)}+\sum^{\infty} _{n=2} {\left({\beta}_{\alpha_n(x)}\prod^{n-1} _{j=1} {{p}_{\alpha_j(x)}}\right)}
\end{equation}
for which
$$
\Delta^{-P} _{\alpha_1\alpha_2...\alpha_n...}=\Delta^{P} _{\alpha_1[s-1-\alpha_2]...\alpha_{2k-1}[s-1-\alpha_{2k}]...},
$$
where
$$
\tilde \delta_{\alpha_n}=\begin{cases}
\sum^{s-1} _{i=s-1-\alpha_{n}} {p_{i}}&\text{if $n$ is even}\\
0&\text{if $n$ is odd and $\alpha_{n}=0$}\\
\sum^{\alpha_n-1} _{i=0}{p_{i}}&\text{if $n$ is odd and  $\alpha_{n}\ne0$},\\
\end{cases}
$$
$$
\tilde  p_{\alpha_n}=\begin{cases}
p_{\alpha_n}&\text{if  $n$ is odd}\\
p_{s-1-\alpha_n}&\text{if $n$ is even},
\end{cases}
$$
and
$$
\beta_{\alpha_n}=\begin{cases}
\sum^{\alpha_n(x)-1} _{i=0} {p_{i}(x)}&\text{whenever $\alpha_n(x)>0$}\\
0&\text{whenever  $\alpha_n(x)=0$.}
\end{cases}
$$
Here $s>1$ is a fixed positive integer and $\alpha_n\in A=\{0,1,\dots, s-1\}$. Also,  $P=\{p_0,p_1,\dots , p_{s-1}\}$ is a fixed set whose elements satisfying the following properties: $p_0+p_1+\dots+p_{s-1}=1$ and $p_i>0$ for all $i=\overline{0,s-1}$. 

\begin{definition}
A representation of form~\eqref{nega}  is called \emph{the nega-P-representation of numbers from~$[0,1]$} (this representation is a partial case of the nega-$\tilde Q$-representation considered in the papers \cite{Serbenyuk2016, Symon2019} and of the quasi-nega-$\tilde Q$-representation (\cite{S.Serbenyuk})).

A representation of form~\eqref{P}  is called \emph{the P-representation of numbers from~$[0,1]$}. This representation is a partial case of some representations noted in \cite{S.Serbenyuk, Symon2019}. 
\end{definition}

 In the present article, the main attention is given to properties of sets of the following form:
\begin{equation*}
\mathbb S_{(-P,u)}=\left\{x: x= \Delta^{-P}_{{\underbrace{u\ldots u}_{\alpha_1-1}} \alpha_1{\underbrace{u\ldots u}_{\alpha_2 -1}}\alpha_2 ...{\underbrace{u\ldots u}_{ \alpha_n -1}}\alpha_n...},   \alpha_n \in A_0=\{1,2,\dots , s-1\}\setminus\{u\},  \right\},
\end{equation*}
where $s>3$ is  a fixed positive integer,  $u=\overline{0,s-1}$, and the parameters $u$ and $s$ are fixed for the fixed set $\mathbb  S_{(-P,u)}$.

Let us consider the P-representation  and the nega-P-representation more detail. Let $s$ be a fixed  positive integer,  $s> 1$,   
and $c_1, c_2,\dots ,c_m$ be an ordered tuple of integers such that $c_i\in\{0,1,\dots ,s-1\}$ for $i=\overline{1,m}$.

\begin{definition} 
{\itshape A P-cylinder (or nega-P-cylinder) of rank $m$ with  base $c_1c_2\ldots c_m$} is a set $\Delta^{P} _{c_1c_2\ldots c_m}$ (or $\Delta^{-P} _{c_1c_2\ldots c_m}$) formed by all numbers of the segment  $[0,1]$ with P-representations (or nega-P-representations) in which the first $m$ digits coincide with $c_1,c_2,\dots ,c_m$, respectively, i.e.,
$$
\Delta^{P} _{c_1c_2\ldots c_m}=\left\{x: x=\Delta^{P} _{\alpha_1\alpha_2\ldots\alpha_n\ldots}, \alpha_j=c_j, j=\overline{1,m}\right\}
$$
or
$$
\Delta^{-P} _{c_1c_2\ldots c_m}=\left\{x: x=\Delta^{-P} _{\alpha_1\alpha_2\ldots\alpha_n\ldots}, \alpha_j=c_j, j=\overline{1,m}\right\}
$$
\end{definition}

\begin{lemma}
Cylinders $\Delta^{P} _{c_1c_2\ldots c_m}$,  $\Delta^{-P} _{c_1c_2\ldots c_m}$ have the following properties: 

\begin{enumerate}

\item an arbitrary cylinder $\Delta^{P} _{c_1c_2\ldots c_m}$ or $\Delta^{-P} _{c_1c_2\ldots c_m}$ is a closed interval;
\item the following relationships hold:
$$
\inf \Delta^{P} _{c_1c_2\ldots c_m}= \Delta^{P} _{c_1c_2\ldots c_m000...},
\sup \Delta^{P} _{c_1c_2\ldots c_m}= \Delta^{P} _{c_1c_2\ldots c_m[s-1][s-1][s-1]...};
$$
$$
\inf \Delta^{- P} _{c_1c_2...c_n}=\begin{cases}
\Delta^{-P} _{c_1c_2...c_n[s-1]0[s-1]0[s-1]...}&\text{if $n$ is odd}\\
\Delta^{-P} _{c_1c_2...c_n0[s-1]0[s-1]0[s-1]...}&\text{if $n$ is even}
\end{cases},
$$
$$
\sup \Delta^{- P} _{c_1c_2...c_n}=\begin{cases}
\Delta^{-P} _{c_1c_2...c_n0[s-1]0[s-1]0[s-1]...}&\text{if $n$ is odd}\\
\Delta^{-P} _{c_1c_2...c_n[s-1]0[s-1]0[s-1]...}&\text{if $n$ is even}
\end{cases};
$$
\item
$$
| \Delta^{P} _{c_1c_2\ldots c_m}|=p_{c_1}p_{c_2}\cdots p_{c_m},~~~~~| \Delta^{-P} _{c_1c_2\ldots c_m}|=\tilde p_{c_1}\tilde p_{c_2}\cdots \tilde p_{c_m}; 
$$
\item
$$
 \Delta^{P} _{c_1c_2\ldots c_mc}\subset  \Delta^{P} _{c_1c_2\ldots c_m}, ~~~~~ \Delta^{-P} _{c_1c_2\ldots c_mc}\subset  \Delta^{-P} _{c_1c_2\ldots c_m};
$$
\item
$$
 \Delta^{P} _{c_1c_2\ldots c_m}=\bigcup^{s-1} _{c=0} { \Delta^{P} _{c_1c_2\ldots c_mc}}, ~~~~~ \Delta^{-P} _{c_1c_2\ldots c_m}=\bigcup^{s-1} _{c=0} { \Delta^{-P} _{c_1c_2\ldots c_mc}};
$$
\item
$$
\lim_{m \to \infty} { |\Delta^{P} _{c_1c_2\ldots c_m}|}=\lim_{m \to \infty} { |\Delta^{-P} _{c_1c_2\ldots c_m}|}=0;
$$
\item
$$
\frac{| \Delta^{P} _{c_1c_2\ldots c_mc_{m+1}}|}{| \Delta^{P} _{c_1c_2\ldots c_m}|}=p_{c_{m+1}}, ~~~~~\frac{| \Delta^{-P} _{c_1c_2\ldots c_mc_{m+1}}|}{| \Delta^{-P} _{c_1c_2\ldots c_m}|}=\tilde p_{c_{m+1}};
$$
\item for any $m\in\mathbb N$
$$
\sup\Delta^{P} _{c_1c_2...c_mc}=\inf  \Delta^{P} _{c_1c_2...c_m[c+1]},~~~\text{where $c \ne s-1$;}
$$
$$
\sup\Delta^{-P} _{c_1c_2...c_{m-1}c}=\inf\Delta^{-P} _{c_1c_2...c_{m-1}[c+1]}~~~ \text{if $m$ is odd}, 
$$
$$
\sup \Delta^{-P} _{c_1c_2...c_{m-1}[c+1]}=\inf \Delta^{-P} _{c_1c_2...c_{m-1}c}, ~~~\text{if $m$ is even};
$$
\item  for an arbitrary $x\in [0,1]$ 
$$
\bigcap^{\infty} _{m=1} {\Delta^{P} _{c_1c_2\ldots c_m}}=x=\Delta^{P} _{c_1c_2\ldots  c_m\ldots}~~~\text{and}~~~\bigcap^{\infty} _{m=1} {\Delta^{-P} _{c_1c_2\ldots c_m}}=x=\Delta^{-P} _{c_1c_2\ldots c_m\ldots}.
$$
\item for any $x_1,x_2\in [0,1]$, the following equalities are true:
$$
x_1=\Delta^{-P} _{\alpha_1\alpha_2...\alpha_n...}\equiv \Delta^{P} _{\alpha_1[s-1-\alpha_2]...\alpha_{2k-1}[s-1-\alpha_{2k}]...},  
$$
$$
x_2=\Delta^{P} _{\alpha_1\alpha_2...\alpha_n...}\equiv\Delta^{-P} _{\alpha_1[s-1-\alpha_2]...\alpha_{2k-1}[s-1-\alpha_{2k}]...}.
$$
\end{enumerate}
\end{lemma}

\begin{definition} A number $x \in[0,1]$ is called   {\itshape P-rational} if 
$$
x=\Delta^{P} _{\alpha_1\alpha_2\ldots\alpha_{n-1}\alpha_n000\ldots}
$$
or
$$
x=\Delta^{P} _{\alpha_1\alpha_2\ldots\alpha_{n-1}[\alpha_n-1][s-1][s-1][s-1]\ldots}.
$$
The  other  numbers in $[0,1]$ are called {\itshape P-irrational}.
\end{definition}

\begin{definition}
 Numbers from some countable subset of $[0,1]$ have two different nega-P-representations, i.e., 
$$
\Delta^{-P} _{\alpha_1\alpha_2...\alpha_{n-1}\alpha_n[s-1]0[s-1]0[s-1]...}=\Delta^{-P} _{\alpha_1\alpha_2...\alpha_{n-1}[\alpha_n-1]0[s-1]0[s-1]...}, ~\alpha_n\ne 0.
$$
These numbers are called  \emph{nega-P-rational}, and other numbers from $[0,1]$ are called  \emph{nega-$P$-irrational}.
\end{definition}

%%%%%%%%%%%%%%%%%%%%%%%%%%%%%%%%%%%%%%%%%%
\section{Sets $\mathbb S_{(-P,u)}$ as images of certain fractals under the action of the Salem type function}
%%%%%%%%%%%%%%%%%%%%%%%%%%%%%%%%%%%%%%%%%%

Let us consider  the sets 
\begin{equation*}
\mathbb S_{(s,u)}=\left\{x: x= \Delta^{s}_{{\underbrace{u\ldots u}_{\alpha_1-1}} \alpha_1{\underbrace{u\ldots u}_{\alpha_2 -1}}\alpha_2 ...{\underbrace{u\ldots u}_{ \alpha_n -1}}\alpha_n...},   \alpha_n \in A_0=\{1,2,\dots , s-1\}\setminus\{u\} \right\}
\end{equation*}
and
\begin{equation*}
\mathbb S_{(-s,u)}=\left\{x: x= \Delta^{-s}_{{\underbrace{u\ldots u}_{\alpha_1-1}} \alpha_1{\underbrace{u\ldots u}_{\alpha_2 -1}}\alpha_2 ...{\underbrace{u\ldots u}_{ \alpha_n -1}}\alpha_n...},  \alpha_n \in A_0=\{1,2,\dots , s-1\}\setminus\{u\}  \right\}
\end{equation*}
where $u=\overline{0,s-1}$, the parameters $u$ and $2<s\in\mathbb N$ are fixed for the fixed sets $\mathbb  S_{(s,u)}$, $\mathbb  S_{(-s,u)}$.

Let $x\in\mathbb S_{(s,u)}$. Let us consider the image 
$$
\tilde y=\tilde F \circ f_l\circ f_+(x),
$$
where 
$$
f_+: x=\Delta^s _{\alpha_1\alpha_2...\alpha_n...}\to \Delta^{-s} _{\alpha_1\alpha_2...\alpha_n...}=y
$$
 is not monotonic on the domain and is a nowhere differentiable function~(\cite{S. Serbenyuk functions with complicated local structure 2013}), $f_l(y)=\frac{1}{s+1}-y$, and $\tilde F$ is the distribution function described earlier. That is, in this paper, the main attention is given to properties and a  local structure of a set of the form:
\begin{eqnarray*}
\mathbb S_{(-P,u)}&=\left\{\tilde y: \tilde y=\tilde F \circ f_l\circ f_+(x), x\in\mathbb S_{(s,u)} \right\}\\
&=\left\{x: x= \Delta^{-P}_{{\underbrace{u\ldots u}_{\alpha_1-1}} \alpha_1{\underbrace{u\ldots u}_{\alpha_2 -1}}\alpha_2 ...{\underbrace{u\ldots u}_{ \alpha_n -1}}\alpha_n...}, \alpha_n \ne u, \alpha_n \ne 0 \right\}. 
\end{eqnarray*}
In other words,
$$
 \Delta^{-P}_{{\underbrace{u\ldots u}_{\alpha_1-1}} \alpha_1{\underbrace{u\ldots u}_{\alpha_2 -1}}\alpha_2 ...{\underbrace{u\ldots u}_{ \alpha_n -1}}\alpha_n...}= \Delta^{P}_{{\underbrace{\tilde u\ldots \tilde u}_{\alpha_1-1}} \tilde \alpha_1{\underbrace{\tilde u\ldots \tilde u}_{\alpha_2 -1}}\tilde \alpha_2 ...{\underbrace{\tilde u\ldots \tilde u}_{ \alpha_n -1}}\tilde \alpha_n...},
$$
where
$$
\tilde \alpha_n=\begin{cases}
\alpha_n &\text{whenever $n$ is odd}\\
s-1-\alpha_n &\text{whenever $n$ is even}
\end{cases}
$$
and
$$
\tilde u=\begin{cases}
u &\text{whenever $u$ is situated  at an odd position in the representation}\\
s-1-u &\text{whenever $u$ is situated at  in an even position in the representation}
\end{cases}.
$$

\begin{remark}
Since properties of the set 
$$
\mathbb S_{(-s,u)}=\left\{x: x= \Delta^{-s}_{{\underbrace{u\ldots u}_{\alpha_1-1}} \alpha_1{\underbrace{u\ldots u}_{\alpha_2 -1}}\alpha_2 ...{\underbrace{u\ldots u}_{ \alpha_n -1}}\alpha_n...},   \alpha_n \ne u, \alpha_n \ne 0 \right\}
$$
(here also $u=\overline{0,s-1}$, the parameters $u$ and $s$ are fixed for the set $\mathbb  S_{(-s,u)}$) were investigated (see~\cite{S. Serbenyuk 2013, {S. Serbenyuk 2017  fractals}, {S. Serbenyuk fractals}}), one can consider the set of images
$$
\{z: z=\tilde F\circ f_l (x),~~~ x\in \mathbb S_{(-s,u)}\}\equiv \mathbb S_{(-P,u)}.
$$
\end{remark}

Let us remark that
$$
\inf {\mathbb S_{(-s,u)}}=\begin{cases}
\Delta^{-s} _{1(02)}&\text{if $u=0$}\\
\Delta^{-s} _{113(12)}&\text{if $u=1$}\\
\Delta^{-s} _{(u2)}&\text{if $u\in\{2,3,\dots , s-1\}$}
\end{cases}
$$
and
$$
\sup {\mathbb S_{(-s,u)}}=\begin{cases}
\Delta^{-s} _{(u2)}&\text{if $u\in\{0,1\}$}\\
\Delta^{-s} _{1(u2)}&\text{if $u\in\{2,3,\dots , s-1\}$}.
\end{cases}
$$
Here $(\alpha)$ is period.

%%%%%%%%%%%%%%%%%%%%%%%%%%%%%%%%%%%%%%%%%%%%%%%%%
\section{Some auxiliary notes}
%%%%%%%%%%%%%%%%%%%%%%%%%%%%%%%%%%%%%%%%%%%%%%%%%

In this section, the main attention is given to notes and calculations which are useful for proving the main results.

Suppose $x\in \mathbb S_{(-s,u)}$. Then consider the set 
$$
\underline{\mathbb S_{(s,u)}}\equiv\left\{\underline{y}:\underline{y}=\Delta^s _{([s-1]0)}+x, ~~~x\in \mathbb S_{(-s,u)}\right\}. 
$$
We have
$$
\inf {\underline{\mathbb S_{(s,u)}}}=\begin{cases}
\Delta^{s} _{[s-2](0[s-3])}&\text{if $u=0$}\\
\Delta^{s} _{[s-2]1[s-4](1[s-3])}&\text{if $u=1$}\\
\Delta^{s} _{([s-1-u]2)}&\text{if $u\in\{2,3,\dots , s-1\}$}
\end{cases}
$$
and
$$
\sup {\underline{\mathbb S_{(s,u)}}}=\begin{cases}
\Delta^{s} _{([s-1-u]2)}&\text{if $u\in\{0,1\}$}\\
\Delta^{s} _{[s-2](u[s-3])}&\text{if $u\in\{2,3,\dots , s-1\}$}.
\end{cases}
$$

Suppose $x\in \mathbb S_{(-s,u)}$. Then consider the set 
$$
\overline{\mathbb S_{(s,u)}}\equiv\left\{\overline{y}:\overline{y}=\Delta^s _{(0[s-1])}-x, ~~~x\in \mathbb S_{(-s,u)}\right\}. 
$$
We obtain
$$
\inf {\overline{\mathbb S_{(s,u)}}}=\begin{cases}
\Delta^{s} _{(u[s-3])}&\text{if $u\in\{0,1\}$}\\
\Delta^{s} _{1([s-1-u]2)}&\text{if $u\in\{2,3,\dots , s-1\}$}
\end{cases}
$$
and
$$
\sup {\overline{\mathbb S_{(s,u)}}}=\begin{cases}
\Delta^{s} _{1([s-1]2)}&\text{if $u=0$}\\
\Delta^{s} _{1[s-2]3([s-2]2)}&\text{if $u=1$}\\
\Delta^{s} _{(u[s-3])}&\text{if $u\in\{2,3,\dots , s-1\}$}
\end{cases}
$$

For the present investigation, the following sets  are auxiliary sets: 
$$
\overline{\mathbb S_{(P,u)}}\equiv\{\overline{z}: \overline{z}=\tilde F(\overline{y}), ~~~\overline{y}\in \overline{\mathbb S_{(s,u)}}\}\equiv \left\{\overline{z}: \overline{z}=\Delta^{P}_{{\underbrace{\tilde u\ldots \tilde u}_{\alpha_1-1}} \tilde \alpha_1{\underbrace{\tilde u\ldots \tilde u}_{\alpha_2 -1}}\tilde \alpha_2 ...{\underbrace{\tilde u\ldots \tilde u}_{ \alpha_n -1}}\tilde \alpha_n...}\right\},
$$
$$
\underline{\mathbb S_{(P,u)}}\equiv\{\underline{z}: \underline{z}=\ddot F(\underline{y}), ~~~\underline{y}\in \underline{\mathbb S_{(s,u)}}\}\equiv \left\{\underline{z}: \underline{z}=\Delta^{P}_{{\underbrace{\ddot  u\ldots \ddot u}_{\alpha_1-1}} \ddot  \alpha_1{\underbrace{\ddot  u\ldots \ddot  u}_{\alpha_2 -1}}\ddot  \alpha_2 ...{\underbrace{\ddot  u\ldots \ddot  u}_{ \alpha_n -1}}\ddot  \alpha_n...}\right\},
$$
where
$$
\ddot \alpha_n=\begin{cases}
\alpha_n &\text{whenever $n$ is even}\\
s-1-\alpha_n &\text{whenever $n$ is odd}
\end{cases}
$$
and
$$
\ddot u=\begin{cases}
u &\text{whenever $u$ is situated  at an even position in the representation}\\
s-1-u &\text{whenever $u$ is situated at  in an odd position in the representation}
\end{cases}.
$$

So, one can note the following statement.
\begin{lemma}
For the sets $\underline{\mathbb S_{(P,u)}}$ and $\overline{\mathbb S_{(P,u)}}$, the following equalities hold: 
$$
\inf {\underline{\mathbb S_{(P,u)}}}=\begin{cases}
\Delta^{P} _{[s-2](0[s-3])}&\text{if $u=0$}\\
\Delta^{P} _{[s-2]1[s-4](1[s-3])}&\text{if $u=1$}\\
\Delta^{P} _{([s-1-u]2)}&\text{if $u\in\{2,3,\dots , s-1\}$}
\end{cases}
$$
and
$$
\sup {\underline{\mathbb S_{(P,u)}}}=\begin{cases}
\Delta^{P} _{([s-1-u]2)}&\text{if $u\in\{0,1\}$}\\
\Delta^{P} _{[s-2](u[s-3])}&\text{if $u\in\{2,3,\dots , s-1\}$},
\end{cases}
$$
$$
\inf {{\mathbb S_{(P,u)}}}\equiv \inf {\overline{\mathbb S_{(P,u)}}}=\begin{cases}
\Delta^{P} _{(u[s-3])}&\text{if $u\in\{0,1\}$}\\
\Delta^{P} _{1([s-1-u]2)}&\text{if $u\in\{2,3,\dots , s-1\}$}
\end{cases}
$$
and
$$
\sup {{\mathbb S_{(-P,u)}}}\equiv \sup {\overline{\mathbb S_{(P,u)}}}=\begin{cases}
\Delta^{P} _{1([s-1]2)}&\text{if $u=0$}\\
\Delta^{P} _{1[s-2]3([s-2]2)}&\text{if $u=1$}\\
\Delta^{P} _{(u[s-3])}&\text{if $u\in\{2,3,\dots , s-1\}$}
\end{cases}
$$
\end{lemma}
\begin{proof}
Since $\tilde F$ and $\ddot F$ are continuous and strictly increasing when the inequality $p_i>0$ holds for   all $i=\overline{0,s-1}$
(\cite{Symon2017, Symon2019}), our statement is true.
\end{proof}

%%%%%%%%%%%%%%%
\section{Proof  of Theorem~\ref{th: the first main theorem}}
%%%%%%%%%%%%%%%

One can begin with the following statement. 
\begin{lemma}
An arbitrary set $\mathbb  S_{(-P,u)}$  is an uncountable set.
\end{lemma}
\begin{proof} It is known (\cite{S. Serbenyuk 2013}) that the set $\mathbb  S_{(-s,u)}$  is  uncountable. Really, it follows from using the mapping $h$: 
$$
 x= \Delta^{-s}_{{\underbrace{u...u}_{\alpha_1-1}} \alpha_1{\underbrace{u...u}_{\alpha_2 -1}}\alpha_2 ...{\underbrace{u...u}_{ \alpha_n -1}}\alpha_n...} \stackrel{h}{\longrightarrow}
 \Delta^{-s}_{\alpha_1\alpha_2 ...\alpha_n...}=h(x)=y.
$$
It is easy to see that the set $\{y: y= \Delta^{-s}_{\alpha_1\alpha_2 ...\alpha_n...}, \alpha_n\in A\setminus\{0,u\}\}$ is   an uncountable set.

In our case, we have 
$$
\mathbb  S_{(-P,u)}\equiv  \overline{\mathbb  S_{(P,u)}} \ni \tilde y=\tilde F\circ f_l(x),~~~x\in \mathbb  S_{(-s,u)}.
$$
Since the functions $\tilde F, f_l$ are continuous and monotonic functions determined on $[0,1]$, we obtain that $\mathbb  S_{(-P,u)}$  is  uncountable. 
\end{proof}

Let $P=\{p_0,p_1, \dots , p_{s-1}\}$ be a fixed set of positive numbers such that $p_0+p_1+\dots + p_{s-1}=1$.

Let us consider the  class  $\Phi$  containing  classes $\Phi_{-P_s}$ of sets  $\mathbb  S_{(-P,u)}$ represented  in the form 
\begin{equation}
%\label{S(s,u)1}
\mathbb S_{(-P,u)}\equiv\left\{x: x= \Delta^{-P}_{{\underbrace{u...u}_{\alpha_1-1}} \alpha_1{\underbrace{u...u}_{\alpha_2 -1}}\alpha_2 ...{\underbrace{u...u}_{ \alpha_n -1}}\alpha_n...},   \alpha_n \ne u, \alpha_n \ne 0 \right\}, 
\end{equation}
where $u=\overline{0,s-1}$, the parameters $u$ and $s$ are fixed for the set $\mathbb  S_{(-P,u)}$. That is, for a fixed positive integer $s>3$, the class $\Phi_{-P_s}$ contains the sets  $\mathbb  S_{(-P,0)}, \mathbb  S_{(-P,1)},\dots,\mathbb  S_{(-P,s-1)}$. 

To investigate topological and metric properties of  $\mathbb  S_{(-P,u)}$, we  study properties of cylinders.  

By $(a_1a_2\ldots a_k)$ denote the period $a_1a_2\ldots a_k$ in the representation of a periodic number.

Let $c_1, c_2,\dots , c_n$ be a fixed ordered tuple of integers such that $c_i\in \overline{A}=A\setminus\{0,u\}$ for $i=\overline{1,n}$.

\begin{definition} 
{\itshape A cylinder of rank $n$ with  base $c_1c_2\ldots c_n$} is a set $\Delta^{(-P,u)} _{c_1c_2\ldots c_n}$ of the form: 
$$
\Delta^{(-P,u)} _{c_1c_2\ldots c_n}=\left\{x: x=\Delta^{-P}_{{\underbrace{u...u}_{c_1-1}} c_1{\underbrace{u...u}_{c_2 -1}}c_2 ...{\underbrace{u...u}_{ c_n -1}}c_n{\underbrace{u...u}_{\alpha_{n+1}-1}}\alpha_{n+1}{\underbrace{u...u}_{\alpha_{n+2}-1}}\alpha_{n+2}...}, \alpha_j=c_j, j=\overline{1,n}\right\}.
$$

By analogy, we have
$$
\Delta^{(P,u)} _{c_1c_2\ldots c_n}=\left\{x: x=\Delta^{P}_{{\underbrace{u...u}_{c_1-1}} c_1{\underbrace{u...u}_{c_2 -1}}c_2 ...{\underbrace{u...u}_{ c_n -1}}c_n{\underbrace{u...u}_{\alpha_{n+1}-1}}\alpha_{n+1}{\underbrace{u...u}_{\alpha_{n+2}-1}}\alpha_{n+2}...}, \alpha_{n+k},c_j \in \overline{A}, j=\overline{1,n}, k\in\mathbb N\right\}.
$$
 \end{definition}

\begin{remark}
It is easy to see that
$$
\Delta^{(-P,u)} _{c_1c_2\ldots c_n}=\Delta^{-P}_{{\underbrace{u...u}_{c_1-1}} c_1{\underbrace{u...u}_{c_2 -1}}c_2 ...{\underbrace{u...u}_{ c_n -1}}c_n}\cap \mathbb S_{(-P,u)}
$$
and
$$
\Delta^{(-P,u)} _{c_1c_2\ldots c_n}=\Delta^{(P,u)} _{\tilde c_1\tilde c_2\ldots \tilde c_n}=\Delta^{-P}_{{\underbrace{\tilde u...\tilde u}_{c_1-1}} \tilde c_1{\underbrace{\tilde u...\tilde u}_{c_2 -1}}\tilde c_2 ...{\underbrace{\tilde u...\tilde u}_{ c_n -1}}\tilde c_n}.
$$
\end{remark}

By definition, put
$$
\tilde p_{u,i}=\begin{cases}
p_u &\text{whenever $i$ is odd}\\
p_{s-1-u} &\text{whenever $i$ is even}
\end{cases},
$$
$$
\tilde \beta_{u,i}=\begin{cases}
\beta_u &\text{whenever $i$ is odd}\\
\beta_{s-1-u} &\text{whenever $i$ is even}
\end{cases},
$$
and
$$
C_{n-1}=\{c_1, c_1+c_2,\dots , c_1+c_2+\dots +c_{n-1}\}.
$$

\begin{lemma} Cylinders $ \Delta^{(-P,u)} _{c_1...c_n} $ have the following properties:
\label{lm: Lemma on cylinders}
\begin{enumerate}
\item
$$
\inf  \Delta^{(-P,u)} _{c_1...c_n}=\begin{cases}
\tau_n+\left(\prod^{n} _{j=1}{\tilde p_{c_j, c_1+...+c_j}}\right)\left(\prod_{\substack{i=\overline{1,c_1+...+c_n-1}\\ i\notin C_{n-1}}}{\tilde p_{u,i}}\right)\inf{\overline{\mathbb S_{(P,u)}}} &\text{if $c_1+\dots +c_n$ is even} \\
\tau_n+\left(\prod^{n} _{j=1}{\tilde p_{c_j, c_1+...+c_j}}\right)\left(\prod_{\substack{i=\overline{1,c_1+...+c_n-1}\\ i\notin C_{n-1}}}{\tilde p_{u,i}}\right)\inf{\underline{\mathbb S_{(P,u)}}} &\text{if $c_1+\dots +c_n$ is odd} \\
\end{cases},
$$
$$
\sup  \Delta^{(-P,u)} _{c_1...c_n}=\begin{cases}
\tau_n+\left(\prod^{n} _{j=1}{\tilde p_{c_j, c_1+...+c_j}}\right)\left(\prod_{\substack{i=\overline{1,c_1+...+c_n-1}\\ i\notin C_{n-1}}}{\tilde p_{u,i}}\right)\sup{\overline{\mathbb S_{(P,u)}}} &\text{if $c_1+\dots +c_n$ is even} \\
\tau_n+\left(\prod^{n} _{j=1}{\tilde p_{c_j, c_1+...+c_j}}\right)\left(\prod_{\substack{i=\overline{1,c_1+...+c_n-1}\\ i\notin C_{n-1}}}{\tilde p_{u,i}}\right)\sup{\underline{\mathbb S_{(P,u)}}} &\text{if $c_1+\dots +c_n$ is odd} \\
\end{cases},
$$
where 
$$
\tau_n=\Delta^P _{{\underbrace{\tilde u...\tilde u}_{c_1-1}} \tilde c_1{\underbrace{\tilde u...\tilde u}_{c_2 -1}}\tilde c_2 ...{\underbrace{\tilde u...\tilde u}_{ c_n -1}}\tilde c_n(0)}.
$$

\item If $d(\cdot) $ is the diameter of a set, then
$$
d\left( \Delta^{(-P,u)} _{c_1...c_n}\right)=\begin{cases}
\left(\prod^{n} _{j=1}{\tilde p_{c_j, c_1+...+c_j}}\right)\left(\prod_{\substack{i=\overline{1,c_1+...+c_n-1}\\ i\notin C_{n-1}}}{\tilde p_{u,i}}\right)d\left({\overline{\mathbb S_{(P,u)}}}\right) &\text{if $c_1+\dots +c_n$ is even} \\
\left(\prod^{n} _{j=1}{\tilde p_{c_j, c_1+...+c_j}}\right)\left(\prod_{\substack{i=\overline{1,c_1+...+c_n-1}\\ i\notin C_{n-1}}}{\tilde p_{u,i}}\right)d\left({\underline{\mathbb S_{(P,u)}}}\right) &\text{if $c_1+\dots +c_n$ is odd} \\
\end{cases}
$$
\item $\frac{d(\Delta^{(-P,u)} _{c_1...c_nc_{n+1}})}{d(\Delta^{(-P,u)} _{c_1...c_n})}=$
$$
=\begin{cases}
p_{s-1-c_{n+1}}\left(\prod^{c_1+...+c_{n+1}-1} _{{i={c_1+c_2+...+c_n+1}}}{\tilde p_{u,i}}\right)&\text{if $c_1+\dots +c_n$, $c_{n+1}$ are  even} \\
p_{c_{n+1}}\left(\prod^{c_1+...+c_{n+1}-1} _{{i={c_1+c_2+...+c_n+1}}}{\tilde p_{u,i}}\right)&\text{if $c_1+\dots +c_n$ is odd,  $c_{n+1}$ is  even} \\
p_{s-1-c_{n+1}}\left(\prod^{c_1+...+c_{n+1}-1} _{{i={c_1+c_2+...+c_n+1}}}{\tilde p_{u,i}}\right)\left({\overline{\mathbb S_{(P,u)}}}/{\underline{\mathbb S_{(P,u)}}}\right) &\text{if $c_1+\dots +c_n$, $c_{n+1}$ are odd } \\
p_{c_{n+1}}\left(\prod^{c_1+...+c_{n+1}-1} _{{i={c_1+c_2+...+c_n+1}}}{\tilde p_{u,i}}\right)\left({\underline{\mathbb S_{(P,u)}}}/{\overline{\mathbb S_{(P,u)}}}\right) &\text{if $c_1+\dots +c_n$ is even, $c_{n+1}$ is odd } \\
\end{cases}.
$$
\item 
$$
  \Delta^{(-P,u)} _{c_1c_2...c_n} =\bigcup_{c\in\overline{A}} { \Delta^{(-P,u)} _{c_1c_2...c_nc}}~~~\forall c_n \in \overline{A},~~~n \in \mathbb N.
$$
\item The following relationships are satisfied: 
\begin{enumerate}
\item if $ u\in \{0,1\}$, then 
$$
\begin{cases}
\inf \Delta^{(-P,u)} _{c_1...c_n[c+1]}> \sup \Delta^{(-P,u)} _{c_1...c_nc}&\text{whenever $c_1+\dots +c_n+c$ is even}\\
$$\\
\inf \Delta^{(-P,u)} _{c_1...c_nc}> \sup \Delta^{(-P,u)} _{c_1...c_n[c+1]}&\text{whenever $c_1+\dots +c_n+c$ is odd}
\end{cases}~~~(c\ne s-1);
$$
\item if  $ u \in \{2,3,\dots ,s-3\}$, then for an odd $c_1+\dots +c_n+c$ 
$$
\begin{cases}
\sup \Delta^{(-P,u)} _{c_1...c_nc}< \inf \Delta^{(P,u)} _{c_1...c_n[c+1]}&\text{for all $c+1\le u$}\\
$$\\
\inf \Delta^{(-P,u)} _{c_1...c_nc}> \sup \Delta^{(-P,u)} _{c_1...c_n[c+1]},&\text{for all $u<c$;}
\end{cases}
$$
if  $ u \in \{2,3,\dots ,s-3\}$, then for an even $c_1+\dots +c_n+c$ 
$$
\begin{cases}
\inf \Delta^{(-P,u)} _{c_1...c_n[c+1]}> \sup \Delta^{(-P,u)} _{c_1...c_nc}&\text{for all $u<c$}\\
$$\\
\inf \Delta^{(-P,u)} _{c_1...c_nc}> \sup \Delta^{(-P,u)} _{c_1...c_n[c+1]}&\text{for all $c+1\le u$}
\end{cases};
$$

\item if $ u  \in \{s-2,s-1\}$, then
$$
\begin{cases}
\inf \Delta^{(-P,u)} _{c_1...c_n[c+1]}> \sup \Delta^{(-P,u)} _{c_1...c_nc}&\text{whenever $c_1+\dots +c_n+c$ is odd}\\
\inf \Delta^{(-P,u)} _{c_1...c_nc}> \sup \Delta^{(-P,u)} _{c_1...c_n[c+1]}&\text{whenever $c_1+\dots +c_n+c$ is even}
\end{cases}.
$$
\end{enumerate}
\end{enumerate}
\end{lemma}
\begin{proof} \emph{The first property} follows from equality~\eqref{nega} and the definition of the set $\mathbb S_{(-P,u)}$. 
\emph{The second property} follows from the first property, and \emph{the third property} is a corollary of the first and second properties. \emph{Property 4} follows from the definition of the set. 

Let us prove \emph{Property 5}.  By definition, put
$$
P_n=\prod^{n} _{j=1}{\tilde p_{c_j, c_1+...+c_j}},~~~~~P^{(u)} _{c_1+...+c_n+c-1}=\prod_{\substack{i=\overline{1,c_1+...+c_n+c-1}\\ i\notin C_{n}}}{\tilde p_{u,i}}.
$$

Let $c_1+c_2+\dots +c_n+c$ be an even number and $u\in\{0,1\}$. Then 
$$
\inf \Delta^{(-P,u)} _{c_1...c_n[c+1]}- \sup \Delta^{(-P,u)} _{c_1...c_nc}=P_n\cdot P^{(u)} _{c_1+...+c_n+c-1}\cdot(\tilde \beta_{u, c_1+...+c_n+c}+\tilde \beta_{c+1, c_1`+...+c_n+c+1}\tilde p_{u,c_1+...+c_n+c}
$$
$$
+\tilde p_{c+1, c_1+...+c_n+c+1}\tilde p_{u,c_1+...+c_n+c}\inf{\underline{\mathbb S_{(P,u)}}}-\tilde \beta_{c, c_1+...+c_n+c}-\tilde p_{c, c_1+...+c_n+c}\sup{\overline{\mathbb S_{(P,u)}}})=P_n\cdot P^{(u)} _{c_1+...+c_n+c-1}\times
$$
$$
\times\left(\beta_{s-1-u}+\beta_{c+1}p_{s-1-u}+p_{c+1}p_{s-1-u}\inf{\underline{\mathbb S_{(P,u)}}}-\beta_{s-1-c}-p_{s-1-c}\sup{\overline{\mathbb S_{(P,u)}}}\right)=P_n\cdot P^{(u)} _{c_1+...+c_n+c-1}\times
$$
$$
\times \left(\beta_{c+1}p_{s-1-u}+p_{c+1}p_{s-1-u}\inf{\underline{\mathbb S_{(P,u)}}}-(1-p_{s-1}-\dots -p_{s-c}-p_{s-1-c})+\beta_{s-1-u}-p_{s-1-c}\sup{\overline{\mathbb S_{(P,u)}}}\right)
$$
$$
=P_n\cdot P^{(u)} _{c_1+...+c_n+c-1}\left(\beta_{c+1}p_{s-1-u}+p_{c+1}p_{s-1-u}\inf{\underline{\mathbb S_{(P,u)}}}+p_{s-1-c}(1-\sup{\overline{\mathbb S_{(P,u)}}})+\beta_{s-1-u}-\beta_{s-c}\right)>0
$$
since $u\in\{0,1\}, c>u,$ and $\beta_k=1-p_{s-1}-p_{s-2}-\dots - p_k$.

Let $c_1+c_2+\dots +c_n+c$ be an odd number and $u\in\{0,1\}$. Then 
$$
\inf \Delta^{(-P,u)} _{c_1...c_nc}- \sup \Delta^{(-P,u)} _{c_1...c_n[c+1]}=P_n\cdot P^{(u)} _{c_1+...+c_n+c-1}\cdot(\tilde \beta_{c, c_1+...+c_n+c}+\tilde p_{c, c_1+...+c_n+c}\inf{\underline{\mathbb S_{(P,u)}}}
$$
$$
-\tilde\beta_{u, c_1+...+c_n+c}-\tilde \beta_{c+1, c_1+...+c_n+c+1}\tilde p_{u,c_1+...+c_n+c}-\tilde p_{c+1, c_1+...+c_n+c+1}\tilde p_{u,c_1+...+c_n+c}\sup{\overline{\mathbb S_{(P,u)}}})
$$
$$
=P_n\cdot P^{(u)} _{c_1+...+c_n+c-1}\left(\beta_c+p_c\inf{\underline{\mathbb S_{(P,u)}}}-\beta_u-\beta_{s-c-2}p_u-p_{s-c-2}p_u\sup{\overline{\mathbb S_{(P,u)}}}\right)
$$
$$
=P_n\cdot P^{(u)} _{c_1+...+c_n+c-1}\left(p_c\inf{\underline{\mathbb S_{(P,u)}}}+\beta_c-\beta_u-p_u(1-p_{s-1}-\dots - p_{s-c-1}-p_{s-c-2})-p_{s-c-2}p_u\sup{\overline{\mathbb S_{(P,u)}}}\right)
$$
$$
=P_n\cdot P^{(u)} _{c_1+...+c_n+c-1}\left(p_c\inf{\underline{\mathbb S_{(P,u)}}}+p_{s-c-2}p_u(1-\sup{\overline{\mathbb S_{(P,u)}}})+\beta_c-\beta_{u+1}+p_u(p_{s-c-1}+\dots +p_{s-1})\right)>0
$$
since $\beta_u+p_u=\beta_{u+1}$ and $c>u$, i.e., $c\ge u+1$.

Let us prove the second system of inequalities. Let $c_1+c_2+\dots +c_n+c$ be an odd number and $u\in \{2,3,\dots , s-3\}$. Then for all $c+1\le u$ let us consider the difference
$$
\sup \Delta^{(-P,u)} _{c_1...c_nc}-\inf \Delta^{(-P,u)} _{c_1...c_n[c+1]}=P_n\cdot P^{(u)} _{c_1+...+c_n+c-1}\cdot(\tilde \beta_{c, c_1+...+c_n+c}+\tilde p_{c, c_1+...+c_n+c}\sup{\underline{\mathbb S_{(P,u)}}}
$$
$$
-\tilde \beta_{u, c_1+...+c_n+c}-\tilde \beta_{c+1, c_1+...+c_n+c+1}\tilde p_{u,c_1+...+c_n+c}-\tilde p_{c+1, c_1+...+c_n+c+1}\tilde p_{u,c_1+...+c_n+c}\inf{\overline{\mathbb S_{(P,u)}}})
$$
$$
=P_n\cdot P^{(u)} _{c_1+...+c_n+c-1}\left(\beta_c+p_c\sup{\underline{\mathbb S_{(P,u)}}}-\beta_u-\beta_{s-c-2}p_u-p_{s-c-2}p_u\inf{\overline{\mathbb S_{(P,u)}}}\right)
$$
$$
=P_n\cdot P^{(u)} _{c_1+...+c_n+c-1}\left(\beta_{c+1}-\beta_u-p_c(1-\sup{\underline{\mathbb S_{(P,u)}}})-\beta_{s-c-2}p_u-p_{s-c-2}p_u\inf{\overline{\mathbb S_{(P,u)}}}\right)<0
$$
since $c+1\le u$ and $\beta_c=1-p_{s-1}-p_{s-2}-\dots - p_{c+1}-p_c=\beta_{c+1}-p_c$.

If $c_1+c_2+\dots +c_n+c$ is  an odd number and $u\in \{2,3,\dots , s-3\}$, and $u+1\le c$, then
$$
\inf \Delta^{(-P,u)} _{c_1...c_nc}- \sup \Delta^{(-P,u)} _{c_1...c_n[c+1]}=P_n\cdot P^{(u)} _{c_1+...+c_n+c-1}\cdot(\tilde \beta_{c, c_1+...+c_n+c}+\tilde p_{c, c_1+...+c_n+c}\inf{\underline{\mathbb S_{(P,u)}}}
$$
$$
-\tilde\beta_{u, c_1+...+c_n+c}-\tilde \beta_{c+1, c_1+...+c_n+c+1}\tilde p_{u,c_1+...+c_n+c}-\tilde p_{c+1, c_1+...+c_n+c+1}\tilde p_{u,c_1+...+c_n+c}\sup{\overline{\mathbb S_{(P,u)}}})=
$$
$$
=P_n\cdot P^{(u)} _{c_1+...+c_n+c-1}\left(\beta_c+p_c\inf{\underline{\mathbb S_{(P,u)}}}-\beta_u-\beta_{s-c-2}p_u-p_{s-c-2}p_u\sup{\overline{\mathbb S_{(P,u)}}}\right)
$$
$$
=P_n\cdot P^{(u)} _{c_1+...+c_n+c-1}\left(\beta_c-\beta_{u+1}+p_c\inf{\underline{\mathbb S_{(P,u)}}}+p_{s-c-2}p_u(1-\sup{\overline{\mathbb S_{(P,u)}}})+p_u(p_{s-c-1}+\dots +p_{s-1})\right)>0
$$
since $c\ge u+1$ and $\beta_{s-c-2}=1-p_{s-1}-\dots -p_{s-c-1}-p_{s-c-2}$.

Let us prove the third system of inequalities. Suppose $c_1+c_2+\dots +c_n+c$ is   even  and $u\in \{2,3,\dots , s-3\}$. Then
$$
\inf \Delta^{(-P,u)} _{c_1...c_nc}- \sup \Delta^{(-P,u)} _{c_1...c_n[c+1]}=P_n\cdot P^{(u)} _{c_1+...+c_n+c-1}(\beta_{s-1-c}+p_{s-1-c}\inf{\overline{\mathbb S_{(P,u)}}}-\beta_{s-1-u}-\beta_{c+1}p_{s-1-u}
$$
$$
-p_{s-1-u}p_{c+1}\sup{\underline{\mathbb S_{(P,u)}}})=P_n\cdot P^{(u)} _{c_1+...+c_n+c-1}(p_{s-1-c}\inf{\overline{\mathbb S_{(P,u)}}}+\beta_{s-1-c}-\beta_{s-1-u}-p_{s-1-u}p_{c+1}\sup{\underline{\mathbb S_{(P,u)}}}
$$
$$
-(1-p_{s-1}-\dots -p_{c+2}-p_{c+1})p_{s-1-u})=P_n\cdot P^{(u)} _{c_1+...+c_n+c-1}(p_{s-1-c}\inf{\overline{\mathbb S_{(P,u)}}}+p_{s-1-u}p_{c+1}(1-\sup{\underline{\mathbb S_{(P,u)}}})
$$
$$
+p_{s-1-u}(p_{c+2}+\dots +p_{s-2}+p_{s-1})+\beta_{s-1-c}-\beta_{s-u})>0
$$
since $c+1\le u$, i.e., $s-c-1\ge s-u$.

Let us consider the difference
$$
\sup \Delta^{(-P,u)} _{c_1...c_nc}-\inf \Delta^{(-P,u)} _{c_1...c_n[c+1]}=P_n\cdot P^{(u)} _{c_1+...+c_n+c-1}(\beta_{s-1-c}+p_{s-1-c}\sup{\overline{\mathbb S_{(P,u)}}}-\beta_{s-1-u}-\beta_{c+1}p_{s-1-u}
$$
$$
-p_{c+1}p_{s-1-u}\inf{\underline{\mathbb S_{(P,u)}}})=P_n\cdot P^{(u)} _{c_1+...+c_n+c-1}(-p_{c+1}p_{s-1-u}\inf{\underline{\mathbb S_{(P,u)}}}-\beta_{s-1-u}-\beta_{c+1}p_{s-1-u}
$$
$$
+(1-p_{s-1}-p_{s-2}-\dots - p_{s-1-c})+p_{s-1-c}\sup{\overline{\mathbb S_{(P,u)}}})
$$
$$
=P_n\cdot P^{(u)} _{c_1+...+c_n+c-1}(-p_{c+1}p_{s-1-u}\inf{\underline{\mathbb S_{(P,u)}}}-p_{s-1-c}(1-\sup{\overline{\mathbb S_{(P,u)}}})+\beta_{s-c}-\beta_{s-1-u}-\beta_{c+1}p_{s-1-u})<0
$$
since $u<c$, i.e., $s-1-u\ge s-c$.

Let us prove the 4th system of inequalities. Let $c_1+c_2+\dots +c_n+c$ be an odd number and $u\in \{s-2,s-1\}$. Then
$$
\sup \Delta^{(-P,u)} _{c_1...c_nc}-\inf \Delta^{(-P,u)} _{c_1...c_n[c+1]}=P_n\cdot P^{(u)} _{c_1+...+c_n+c-1}\times
$$
$$
\times\left(\beta_c+p_c\sup{\underline{\mathbb S_{(P,u)}}}-\beta_u-\beta_{s-c-2}p_u-p_{s-c-2}p_u\inf{\overline{\mathbb S_{(P,u)}}}\right)=P_n\cdot P^{(u)} _{c_1+...+c_n+c-1}\times
$$
$$
\times\left(-\beta_{s-c-2}p_u-p_{s-c-2}p_u\inf{\overline{\mathbb S_{(P,u)}}}-\beta_u+p_c\sup{\underline{\mathbb S_{(P,u)}}}+(1-p_{s-1}-p_{s-2}-\dots -p_{c+1}-p_c)\right)
$$
$$
=P_n\cdot P^{(u)} _{c_1+...+c_n+c-1}\left(-\beta_{s-c-2}p_u-p_{s-c-2}p_u\inf{\overline{\mathbb S_{(P,u)}}}-p_c(1-\sup{\underline{\mathbb S_{(P,u)}}})+\beta_{c+1}-\beta_u\right)<0
$$
since $u\ge c+1$.

 Suppose $c_1+c_2+\dots +c_n+c$ is even. Then 
$$
\inf \Delta^{(-P,u)} _{c_1...c_nc}- \sup \Delta^{(-P,u)} _{c_1...c_n[c+1]}=P_n\cdot P^{(u)} _{c_1+...+c_n+c-1}\times
$$
$$
\times\left(\beta_{s-1-c}+p_{s-1-c}\inf{\overline{\mathbb S_{(P,u)}}}-\beta_{s-1-u}-\beta_{c+1}p_{s-1-u}-p_{s-1-u}p_{c+1}\sup{\underline{\mathbb S_{(P,u)}}}\right)=P_n\cdot P^{(u)} _{c_1+...+c_n+c-1}\times
$$
$$
\times \left(p_{s-1-c}\inf{\overline{\mathbb S_{(P,u)}}}+\beta_{s-1-c}-\beta_{s-1-u}-p_{s-1-u}(1-p_{s-1}-\dots -p_{c+1})-p_{s-1-u}p_{c+1}\sup{\underline{\mathbb S_{(P,u)}}}\right)
$$
$$
=P_n\cdot P^{(u)} _{c_1+...+c_n+c-1}(p_{s-1-c}\inf{\overline{\mathbb S_{(P,u)}}}+\beta_{s-1-c}-\beta_{s-1-u}-p_{s-1-u}+p_{s-1-u}p_{c+1}(1-\sup{\underline{\mathbb S_{(P,u)}}})
$$
$$
+p_{s-1-u}(p_{c+2}+\dots +p_{s-2}+p_{s-1}))>0
$$
since $u+1\ge c$, i.e., $s-u\le s-c+1$, and $\beta_{s-1-c}-\beta_{s-1-u}-p_{s-1-u}=\beta_{s-1-c}-\beta_{s-u}\ge 0$.
\end{proof}

 Let us  prove that \emph{the set $\mathbb  S_{(-P,u)}$ is a  nowhere dense set}.  From the definition, it follows that there exist cylinders $ \Delta^{(-P,u)} _{c_1...c_n}$ of rank $n$ in an arbitrary subinterval of the segment    $I=[\inf\mathbb  S_{(-P,u)},\sup\mathbb  S_{(-P,u)}]$. Since Property 5 from Lemma~\ref{lm: Lemma on cylinders} is true  for these cylinders, we have that for any subinterval of  $ I$ there exists a subinterval such that does not contain points from $\mathbb  S_{(-P,u)}$. So $\mathbb  S_{(-P,u)}$ is a  nowhere dense set.

Let us show that \emph{$\mathbb  S_{(-P,u)}$ is a set of zero Lebesgue measure}. Suppose that $ I^{(-P,u)} _{c_1c_2...c_n} $ is a closed interval whose endpoints coincide with endpoits of the cylinder $ \Delta^{(-P,u)} _{c_1c_2...c_n}$. It is easy to see that 
$$
|I^{(-P,u)} _{c_1c_2...c_n}|=d(\Delta^{(-P,u)} _{c_1c_2...c_n}).
$$
Also, 
$$
 \mathbb  S_{(-P,u)}= \bigcap^{\infty} _{k=1}{S_k},
$$
where
$$
S_1=\bigcup_{c_1\in \overline{A}=A_0\setminus\{u\}}{I^{(-P,u)} _{c_1}},
$$
$$
S_2=\bigcup_{c_1,c_2\in \overline{A}}{I^{(-P,u)} _{c_1c_2}},
$$
$$
\dots\dots\dots\dots\dots\dots\dots
$$
$$
S_k= \bigcup_{c_1,c_2,...,c_k\in \overline{A}}{I^{(-P,u)} _{c_1c_2...c_k}},
$$
$$
\dots\dots\dots\dots\dots\dots\dots
$$
In addition, since $S_{k+1} \subset S_k $, we have 
$$
S_k= S_{k+1} \cup \bar S_{k+1}.
$$

Suppose 
$$
I_{\overline{0}}=\left[\inf{\overline{\mathbb S_{(P,u)}}},\sup{\overline{\mathbb S_{(P,u)}}}\right]
$$
and
$$
I_{\underline{0}}=\left[\inf{\underline{\mathbb S_{(P,u)}}},\sup{\underline{\mathbb S_{(P,u)}}}\right]
$$
are initial closed intervals,  $\lambda(\cdot)$ is the Lebesgue measure of a set.  Then
$$
0<\lambda(S_1)=\sum_{\substack{c_1 \text{is odd}\\ c_1\in \overline{A}}}{\left(\tilde p_{c_1,c_1}\prod^{c_1-1} _{i=1}{\tilde p_{u,i}}\right)}\lambda(I_{\underline{0}})+\sum_{\substack{c_1 \text{is even}\\ c_1\in \overline{A}}}{\left(\tilde p_{c_1,c_1}\prod^{c_1-1} _{i=1}{\tilde p_{u,i}}\right)}\lambda(I_{\overline{0}}).
$$
It is easy to see that 
$$
0<\lambda(S_1)\le \sum_{ c_1\in \overline{A}}{\left(\tilde p_{c_1,c_1}\prod^{c_1-1} _{i=1}{\tilde p_{u,i}}\right)}\max{\left\{\lambda(I_{\overline{0}}), \lambda(I_{\underline{0}})\right\}}<1
$$
since 
$$
\lambda\left(\bigcup_{c_1,...,c_n\in A=\{0,1,...,s-1\}}{\Delta^P _{c_1c_2...c_n}}\right)=1. 
$$
Also, one can denote
$$
v_{c_1}=\sum_{ c_1\in \overline{A}}{\left(\tilde p_{c_1,c_1}\prod^{c_1-1} _{i=1}{\tilde p_{u,i}}\right)}<1=p_0+p_1+\dots +p_{s-1}.
$$

In the second step, we get
\begin{equation*}
\begin{split}
0<\lambda(S_2)&=\sum_{\substack{c_1+c_2 \text{is odd}\\ c_1,c_2\in \overline{A}}}{\left(\tilde p_{c_1,c_1}\tilde p_{c_2,c_1+c_2}\prod _{\substack{i=\overline{1, c_1+c_2-1}\\ i\ne c_1}}{\tilde p_{u,i}}\right)}\lambda(I_{\underline{0}})\\
&+\sum_{\substack{c_1+c_2 \text{is even}\\ c_1,c_2\in \overline{A}}}{\left(\tilde p_{c_1,c_1}\tilde p_{c_2,c_1+c_2}\prod _{\substack{i=\overline{1, c_1+c_2-1}\\ i\ne c_1}}{\tilde p_{u,i}}\right)}\lambda(I_{\overline{0}})\\
&\le\sum_{ c_1,c_2\in \overline{A}}{\left(\tilde p_{c_1,c_1}\tilde p_{c_2,c_1+c_2}\prod _{\substack{i=\overline{1, c_1+c_2-1}\\ i\ne c_1}}{\tilde p_{u,i}}\right)}\max{\left\{\lambda(I_{\overline{0}}), \lambda(I_{\underline{0}})\right\}}\\
&\le \max{\left\{\lambda(I_{\overline{0}}), \lambda(I_{\underline{0}})\right\}}\left(\max{\left\{v_{c_1}, v_{c_2}\right\}}\right)^2,
\end{split}
\end{equation*}
where
$$
v_{c_2}=\sum_{ c_2\in \overline{A}}{\left(\tilde p_{c_2,c_1+c_2}\prod^{c_1+c_2-1} _{i=c_1+1}{\tilde p_{u,i}}\right)}<1.
$$

In the $n$th step, we have
\begin{equation*}
\begin{split}
0<\lambda(S_n)&=\sum_{\substack{c_1+c_2...+c_n \text{is odd}\\ c_1,c_2,...,c_n\in \overline{A}}}{\left(\left(\prod^{n} _{j=1}{\tilde p_{c_j, c_1+c_2+...+c_j}}\right)\prod _{\substack{i=\overline{1, c_1+c_2+...+c_n-1}\\ i\ne C_{n-1}}}{\tilde p_{u,i}}\right)}\lambda(I_{\underline{0}})\\
&+\sum_{\substack{c_1+c_2...+c_n \text{is even}\\ c_1,c_2,...,c_n\in \overline{A}}}{\left(\left(\prod^{n} _{j=1}{\tilde p_{c_j, c_1+c_2+...+c_j}}\right)\prod _{\substack{i=\overline{1, c_1+c_2+...+c_n-1}\\ i\ne C_{n-1}}}{\tilde p_{u,i}}\right)}\lambda(I_{\overline{0}})\\
&\le\sum_{c_1,c_2,...,c_n\in \overline{A}}{\left(\left(\prod^{n} _{j=1}{\tilde p_{c_j, c_1+c_2+...+c_j}}\right)\prod _{\substack{i=\overline{1, c_1+c_2+...+c_n-1}\\ i\ne C_{n-1}}}{\tilde p_{u,i}}\right)}\max{\left\{\lambda(I_{\overline{0}}), \lambda(I_{\underline{0}})\right\}}\\
&\le \max{\left\{\lambda(I_{\overline{0}}), \lambda(I_{\underline{0}})\right\}}\left(\max_{k=\overline{1,n}}{\left\{\sum_{ c_k\in \overline{A}}{\left(\tilde p_{c_k,c_1+c_2+...+c_k}\prod^{c_1+c_2+...+c_k-1} _{i=c_1+c_2+...+c_{k-1}+1}{\tilde p_{u,i}}\right)}\right\}}\right)^n\\
&=\max{\left\{\lambda(I_{\overline{0}}), \lambda(I_{\underline{0}})\right\}}\left(\max_{k=\overline{1,n}}{\left\{v_{c_k}\right\}}\right)^n<1.\\
\end{split}
\end{equation*}
Here $c_{k-1}=0$ for $k=1$.

So,
$$
\lim_{n\to\infty}{\lambda(S_n)}\le \lim_{n\to\infty}{\left(\max{\left\{\lambda(I_{\overline{0}}), \lambda(I_{\underline{0}})\right\}}\left(\max_{k=\overline{1,n}}{\left\{v_{c_k}\right\}}\right)^n\right)}=0.
$$
Hence the set $\mathbb  S_{(-P,u)}$  is a set of zero Lebesgue measure. 

Let us prove that \emph{$\mathbb  S_{(-P,u)}$  is a perfect set}. Since 
$$ 
S_k= \bigcup_{c_1,c_2,...,c_k\in \overline{A}}{I^{(-P,u)}  _{c_1c_2...c_k}}
$$
 is a closed set ($S_k$ is a union of segments), we see that 
$$
 \mathbb  S_{(-P,u)}= \bigcap^{\infty} _{k=1} S_k
$$
is a closed set. 

Suppose $ x \in \mathbb  S_{(-P,u)} $,    $ R$  is any interval  containing $ x $, and $ J_n $ is a segment of  $ S_n $ such that contains $ x $. Let us choose a number $ n $ such that $  J_n \subset R $. Suppose that $ x_n $ is the endpoint of $ J_n $ such that the condition 
$ x_n \ne x $ holds. Hence $ x_n \in \mathbb  S_{(-P,u)} $ and $  x $ is a limit point of the set. 

Since $\mathbb  S_{(-P,u)}$ is a closed set and does not contain isolated points, we obtain that $\mathbb  S_{(-P,u)}$ is a perfect set.

%%%%%%%%%%%%%%%%%%%%%%%%%%%
\section{Proof of Theorem~\ref{th: the second main theorem}}
%%%%%%%%%%%%%%%%%%%%%%%%%%%

Since 
$ \mathbb S_{(-P,u)} \subset I_{\overline{0}}$ and $ \mathbb S_{(-P,u)}$ is a perfect set, we obtain that  $\mathbb S_{(-P,u)}$ is a compact set.
In addition, 
$$
\frac{\Delta^{(-P,u)} _{c_1c_2...c_{n-1}c_n}}{\Delta^{(-P,u)} _{c_1c_2...c_{n-1}}}=\begin{cases}
\omega_{1,c_n}=\underbrace{p_{s-1-u}p_u\ldots p_{s-1-u}p_u}_{c_n-1}p_{s-1-c_n}\frac{d\left(\overline{\mathbb S_{(P,u)}}\right)}{d\left(\underline{\mathbb S_{(P,u)}}\right)}&\text{if  $c_1+\dots +c_{n-1}$ is odd, $c_n$ is odd}\\
\omega_{2,c_n}=\underbrace{p_up_{s-1-u}\ldots p_up_{s-1-u}}_{c_n-1}p_{c_n}\frac{d\left(\underline{\mathbb S_{(P,u)}}\right)}{d\left(\overline{\mathbb S_{(P,u)}}\right)}&\text{if  $c_1+\dots +c_{n-1}$ is even, $c_n$ is odd}\\
\omega_{3,c_n}=\underbrace{p_{s-1-u}p_u\ldots p_{s-1-u}p_up_{s-1-u}}_{c_n-1}p_{c_n}&\text{if  $c_1+\dots +c_{n-1}$ is odd, $c_n$ is even}\\
\omega_{4,c_n}=\underbrace{p_up_{s-1-u}\ldots p_up_{s-1-u}p_{u}}_{c_n-1}p_{s-1-c_n}&\text{if  $c_1+\dots +c_{n-1}$ is even, $c_n$ is even}
\end{cases}
$$
and
$$
\mathbb S_{(-P,u)}=\bigcap^{\infty} _{n=1}{\bigcup_{c_1,\dots , c_n\in \overline{A}}{\Delta^{(-P,u)} _{c_1c_2\ldots c_n}}}. 
$$

Suppose $l$ is the number of odd numbers in the set  $\overline{A}=\{0,1, \dots , s-1\}\setminus\{0,u\}$ and $m$ is the number of even numbers in $\overline{A}$. We have $(l+m)^n$ cylinders $\Delta^{(-P,u)} _{c_1c_2...c_{n-1}c_n}$.

Let $\check N_{j,n}$ ($j=\overline{1,4}, 1<n\in\mathbb N$) be the summary number  of cylinders $\Delta^{(-P,u)} _{c_1c_2...c_n}$ for which
$$
\frac{d\left(\Delta^{(-P,u)} _{c_1c_2...c_{n-1}c_n}\right)}{d\left(\Delta^{(-P,u)} _{c_1c_2...c_{n-1}}\right)}=\omega_{j,c_n}.
$$

\emph{In the first step}, we have $\check N_{2,1}=l$ (the unique cylinder for any odd $c_1\in \overline{A}$) and $\check N_{4,1}=m$ (the unique cylinder for any even $c_1\in \overline{A}$) because for $n=0$ we get $I_{\overline{0}}\cap \overline{\mathbb S_{(P,u)}}$. 

\emph{In the second step}, we have:
\begin{itemize}
\item $\check N_{1,2}=l^2$ ($l$ cylinders for any odd $c_2\in \overline{A}$); 
\item $\check N_{2,2}=lm$ ($m$ cylinders for any odd $c_2\in \overline{A}$);
\item $\check N_{3,2}=lm$ ($l$ cylinders for any even $c_2\in \overline{A}$);
\item $\check N_{4,2}=m^2$ ($m$ cylinders for any even $c_2\in \overline{A}$).
\end{itemize}

\emph{In the third step}, we have:
\begin{itemize}
\item $\check N_{1,3}=2l^2m$ ($2lm$ cylinders for any odd $c_3\in \overline{A}$); 
\item $\check N_{2,3}=l(l^2+m^2)$ ($l^2+m^2$ cylinders for any odd $c_3\in \overline{A}$);
\item $\check N_{3,3}=2lm^2$ ($2lm$ cylinders for any even $c_3\in \overline{A}$);
\item $\check N_{4,3}=m(l^2+m^2)$ ($l^2+m^2$ cylinders for any even $c_3\in \overline{A}$).
\end{itemize}

\emph{In the fourth step}, we have:
\begin{itemize}
\item $\check N_{1,4}=l^4+3l^2m^2$ ($l^3+3lm^2$ cylinders for any odd $c_4\in \overline{A}$); 
\item $\check N_{2,4}=lm^3+3l^3m$ ($m^3+3l^2m$ cylinders for any odd $c_4\in \overline{A}$);
\item $\check N_{3,4}=l^3m+3lm^3$ ($l^3+3lm^2$ cylinders for any even $c_4\in \overline{A}$);
\item $\check N_{4,4}=m^4+3l^2m^2$ ($m^3+3l^2m$ cylinders for any even $c_4\in \overline{A}$).
\end{itemize}

Let us remark that, in the general case, values of $\omega_{j,c_n}$ are different for  the unique $j$ but different $c_n\in\overline{A}.$ Hence the Hausdorff dimension of our set depends on the numbers of $\omega_{j,c_n}$  for all different  $c_n\in\overline{A}$. 

Using arguments described in \cite{HRW2000} and auxiliary Theorems~\ref{th: auxiliary-1}--\ref{th: auxiliary-4}, this completes the proof.


\begin{thebibliography}{9}

\bibitem{ACFS2017}
E.~de Amo,  M.D.~Carrillo and  J.~Fern\'andez-S\'anchez, A Salem generalized function, \emph{Acta Math. Hungar. }  \textbf{151} (2017), no.~2,  361--378. https://doi.org/10.1007/s10474-017-0690-x

\bibitem{BBLS2016}
A. S. Balankin, J. Bory Reyes, M. E. Luna-Elizarrar\'as and M. Shapiro, Cantor-type sets in
hyperbolic numbers, \emph{Fractals} \textbf{24} (2016), no. 4, Paper No. 1650051.

\bibitem{Broderick2011}
R. Broderick,  L. Fishman, and A. Reich,  Intrinsic Approximation on Cantor-like Sets, a Problem of Mahler, \emph{Moscow Journal of Combinatorics and Number Theory} \textbf{1} (2011), no. 4, 3--12. 

\bibitem{Bunde1994}
A. Bunde and S. Havlin,\emph{ Fractals in Science},
Springer, New York, 1994.

\bibitem{Dani2012}
S. G. Dani and Hemangi Shah, Badly approximable numbers and vectors in Cantor-like sets, \emph{Proc. Amer. Math. Soc.} \textbf{140} (2012), 2575--2587. 

\bibitem{DU2014} R. DiMartino and W. O. Urbina, On Cantor-like sets and Cantor-Lebesgue singular functions, https://arxiv.org/pdf/1403.6554.pdf

\bibitem{DU2014(2)} R. DiMartino and W. O. Urbina, Excursions on Cantor-like Sets, https://arxiv.org/pdf/1411.7110.pdf

\bibitem{Falconer1997} K.  Falconer  \emph {Techniques in Fractal Geometry}, John Willey and Sons, 1997.

\bibitem{Falconer2004} K. Falconer  \emph{Fractal Geometry: Mathematical Foundations and Applications},  Wiley,  2004.

\bibitem{Feng2005} D.J. Feng, The limited Rademacher functions and Bernoulli convolutions associated
with Pisot numbers, \emph{ Adv. Math.} \textbf{195} (2005), 24--101. 

\bibitem{HRW2000} S. Hua,  H. Rao, Z. Wen et al., On the structures and dimensions of Moran sets, \emph{Sci. China Ser. A-Math.} \textbf{43} (2000), No. 8,  836--852 

\bibitem{Ito2009} S. Ito and T. Sadahiro, Beta-expansions with negative bases, \emph{ Integers} \textbf{ 9} (2009),   239--259.

\bibitem{KLS2016} A.~K\"aenm\"aki, B.~Li, and V.~Suomala,  Local dimensions in Moran constructions, \emph{Nonlinearity} \textbf{29} (2016), No.~3, 807--822. 


\bibitem{KKK1990} S. Kalpazidou, A. Knopfmacher, and  J. Knopfmacher, L\"uroth-type alternating series 
representations for real numbers, \emph{  Acta Arithmetica} \textbf{ 55} (1990),  311--322.

\bibitem{Katsuura1991}	H. Katsuura, Continuous Nowhere-Differentiable Functions - an Application of Contraction Mappings, \emph{The American Mathematical Monthly} \textbf{98} (1991), no. 5, 411--416, https://doi.org/10.1080/00029890.1991.12000778     

\bibitem{Kawamura2010}
Kiko Kawamura, The derivative of Lebesgue's singular function, \emph{ Real Analysis Exchange}
Summer Symposium 2010, pp. 83--85.

\bibitem{Kennedy1995}
J. A. Kennedy and J. A. Yorke, Bizarre topology is
natural in dynamical systems, \emph{ Bull. Am. Math. Soc.
(N.S.)} \textbf{32}  (1995), no. 3,  309--316.


\bibitem{LW2011}
J. Li and M. Wu, Pointwise dimensions of general Moran measures with open set condition, 
\emph{Sci. China, Math.}  \textbf{54} (2011), 699--710.

\bibitem{Mandelbrot1977} B.~B. Mandelbrot, \emph{Fractals: Form, Chance and Dimension}, Freeman, San Francisco,  1977. 

\bibitem{Mandelbrot1999}
B. Mandelbrot, \emph{The Fractal Geometry of Nature},
18th printing,  Freeman, New York, 1999.

\bibitem{Moran1946}  P. A. P. Moran, Additive functions of intervals and Hausdorff measure, \emph{Mathematical Proceedings of the
Cambridge Philosophical Society} \textbf{42} (1946), No. 1, 15--23  doi:10.1017/S0305004100022684

\bibitem{Palis1993}
J. Palis and F. Takens, \emph{Hyperbolicity and Sensitive Chaotic Dynamics at Homoclinic Bifurcations:
Fractal Dimensions and Infinitely Many Attractors, Cambridge Studies in Advanced Mathematics},
Vol. 35 (Cambridge University Press, Cambridge, 1993).

\bibitem{PS1995}  M. Pollicott and K\'aroly Simon, The Hausdorff dimension of $\lambda$-expansions with deleted digits, 
 \emph{Trans. Amer. Math. Soc.}  \textbf{347} (1995), 967-983 https://doi.org/10.1090/S0002-9947-1995-1290729-0 

\bibitem{Renyi1957} A.  R\'enyi, Representations for real numbers and their ergodic properties, \emph{Acta. 
  Math. Acad. Sci. Hungar.} \textbf{8} (1957),  477--493.


\bibitem{Salem1943}{ R. Salem},
 {On some singular monotonic functions which are stricly increasing,}
 {\itshape Trans. Amer. Math. Soc.} {\bf 53} (1943), 423--439.


\bibitem{S. Serbenyuk abstract 2} S. O. Serbenyuk, Topological, metric and fractal properties of one set defined by using the s-adic representation, \emph{XIV International Scientific Kravchuk Conference:} Conference materials II, Kyiv: National Technical University of Ukraine ``KPI", 2012. --- P. 220 (in Ukrainian), available at  https://www.researchgate.net/publication/311665455

\bibitem{S. Serbenyuk abstract 3} S. O. Serbenyuk, Topological, metric and fractal properties of sets of class generated by one set with using  the s-adic representation, \emph{International Conference ``Dynamical Systems and their Applications":} Abstracts, Kyiv: Institute of Mathematics of NAS of Ukraine, 2012. --- P. 42 (in Ukrainian),  available at
https://www.researchgate.net/publication/311415778


\bibitem{S. Serbenyuk abstract 5} S. O. Serbenyuk,  Topological, metric and fractal properties of the set with parameter, that the set defined by s-adic representation of numbers,  \emph{International Conference ``Modern Stochastics: Theory and Applications III"  dedicated to 100th anniversary of    B.~V.~Gnedenko and 80th anniversary of  M.~I.~Yadrenko:} Abstracts, Kyiv: Taras Shevchenko National University of Kyiv, 2012. --- P.~13,  available at 
https://www.researchgate.net/publication/311415501


\bibitem{Symon1}  S. O. Serbenyuk, Topological, metric, and fractal properties of one set of real numbers such that it defined in terms of the s-adic representation, \emph{Naukovyi Chasopys NPU im. M.~P.~Dragomanova. Seria~1. Phizyko-matematychni Nauky}[\emph{Trans. Natl. Pedagog. Mykhailo Dragomanov University. Ser.~1. Phys. Math.}]  \textbf{11}   (2010), 241--250.  (in Ukrainian), available at  https://www.researchgate.net/publication/292606441

\bibitem{Symon2}  S. O. Serbenyuk, Topological, metric properties and using one generalizad set determined by the s-adic representation with a parameter, \emph{Naukovyi Chasopys NPU im. M.~P.~Dragomanova. Seria~1. Phizyko-matematychni Nauky}[\emph{Trans. Natl. Pedagog. Mykhailo Dragomanov University. Ser.~1. Phys. Math.}]  \textbf{12}   (2011), 66--75.  (in Ukrainian), available at https://www.researchgate.net/publication/292970196

 \bibitem{S. Serbenyuk 2013} 
 S. O. Serbenyuk, On some sets of real numbers such  that defined  by nega-s-adic and Cantor nega-s-adic representations, \emph{Trans. Natl. Pedagog. Mykhailo Dragomanov Univ. Ser. 1. Phys. Math.} 15 (2013), 168-187,  available at https://www.researchgate.net/publication/292970280 (in Ukrainian)

\bibitem{Symon2015}{\itshape S. O. Serbenyuk},
 {Functions, that defined by functional equations systems in terms of Cantor series representation of numbers,}
 {\itshape Naukovi Zapysky NaUKMA} {\bf 165} (2015), 34--40. (Ukrainian),  available at https://www.researchgate.net/publication/292606546

%\bibitem{preprint1-2018}
%Symon Serbenyuk, Generalizations of certain representations of real numbers, arXiv:1801.10540v3, 8 pp.

\bibitem{S.Serbenyuk} \emph{Serbenyuk S.} On some generalizations of real numbers representations, arXiv:1602.07929v1 (in Ukrainian)

 \bibitem{S. Serbenyuk 2017  fractals} S. Serbenyuk, One one class of fractal sets, https://arxiv.org/pdf/1703.05262.pdf

\bibitem{S. Serbenyuk fractals} S. Serbenyuk, More on one class of fractals, arXiv:1706.01546v1.


\bibitem{Symon} S. Serbenyuk, One distribution function on the Moran sets, arXiv:1808.00395v1. 

\bibitem{Serbenyuk2016}
S.~ Serbenyuk. Nega-$\tilde Q$-representation as a generalization of certain alternating representations
of real numbers, \emph{Bull. Taras Shevchenko Natl. Univ. Kyiv Math. Mech.} 1 (35) (2016), 32-39,
available at https://www.researchgate.net/publication/308273000 (in Ukrainian)

\bibitem{S. Serbenyuk functions with complicated local structure 2013}{\itshape S. Serbenyuk},
 On one class of functions with complicated local structure, \emph{\v{S}iauliai Mathematical Seminar} {\bf 11 (19)} (2016), 75--88.

\bibitem{Symon2017} S. O. Serbenyuk, Continuous Functions with Complicated Local Structure Defined in Terms of Alternating Cantor Series Representation of Numbers, {\itshape Zh. Mat. Fiz. Anal. Geom.} 
 {\bf 13} (2017), no. 1,  57--81.


\bibitem{S. Serbenyuk preprint2} \emph{Serbenyuk S.} Non-differentiable functions defined in terms of classical representations of real numbers, {\itshape Zh. Mat. Fiz. Anal. Geom.} 
 {\bf 14} (2018), no. 2,  197--213.

\bibitem{S. Serbenyuk abstract1} S. O. Serbenyuk,  Preserving the Hausdorff-Besicovitch dimension by monotonic singular distribution
functions. In: Second interuniversity scientific conference on mathematics and physics for young scientists:
abstracts, pp. 106-107. Institute of Mathematics of NAS of Ukraine, Kyiv (2011). https://www.
researchgate.net/publication/301637057 (in Ukrainian)

\bibitem{S.Serbenyuk 2017} S. Serbenyuk, On one fractal property of the Minkowski function, \emph{Revista de la Real Academia de Ciencias Exactas, F\' isicas y Naturales. Serie A. Matem\' aticas} \textbf{112} (2018), no.~2, 555--559, doi:10.1007/s13398-017-0396-5

\bibitem{Symon2019} Symon Serbenyuk, On one application of infinite systems of functional equations in function theory, \emph{ Tatra Mountains Mathematical Publications} \textbf{74} (2019), 117-144. https://doi.org/10.2478/tmmp-2019-0024  


\bibitem{Taylor2012}
T. D. Taylor, C. Hudson, and A. Anderson, Examples of using binary Cantor sets to study the connectivity of Sierpinski relatives, \emph{Fractals} \textbf{20} (2012), no. 1,  61--75.

\bibitem{TSBR2017}
G.Y. T\'ellez-S\'anchez and J. Bory-Reyes, More about Cantor like sets in hyperbolic numbers, \emph{Fractals} \textbf{25}
 (2017), no. 5, Paper No. 1750046.



\bibitem{WW2008} B.W. Wang and J. Wu,  Hausdorff dimension of certain sets arising in continued fraction
expansions, \emph{ Adv. Math.} \textbf{218} (2008), 1319--1339.


  \bibitem{wiki-fractal}	Wikipedia contributors, "Fractal", Wikipedia, The Free Encyclopedia,  https://en.wikipedia.org/wiki/Fractal   (accessed April 28, 2020).

\bibitem{wiki-pathological} Wikipedia contributors, ``Pathological (mathematics)", Wikipedia, The Free Encyclopedia, https://en.wikipedia.org/wiki/Pathological\_(mathematics)  (accessed April 28, 2020). 



\bibitem{wiki-singular}	Wikipedia contributors, "Singular function", Wikipedia, The Free Encyclopedia,  https://en.wikipedia.org/wiki/Singular\_function    (accessed April 28, 2020).  

\bibitem{wiki-function}	Wikipedia contributors, "Thomae's function", Wikipedia, The Free Encyclopedia,  https://en.wikipedia.org/wiki/Thomae's\_function   (accessed April 28, 2020).  

\bibitem{Wu2005} J. Wu,  On the sum of degrees of digits occurring in continued fraction expansions of
Laurent series, \emph{ Math. Proc. Camb. Philos. Soc.} \textbf{138} (2005), 9--20.

\bibitem{W2005} 
M.Wu, The singularity spectrum $f(\alpha)$ of some Moran fractals, \emph{Monatsh. Math.} \textbf{144}
(2005), 141--155.


\bibitem{wiki-self-similarity}	Wikipedia contributors, "Self-similarity", Wikipedia, The Free Encyclopedia,     
https://en.wikipedia.org/wiki/Self-similarity (accessed August 17, 2020). 


\end{thebibliography}
\end{document}